\documentclass[a4paper]{article}

\usepackage[all]{xy}\usepackage[latin1]{inputenc}        
\usepackage[dvips]{graphics,graphicx}
\usepackage{amsfonts,amssymb,amsmath,xcolor,mathrsfs, amstext}
\usepackage{amsbsy, amsopn, amscd, amsxtra, amsthm,authblk, enumerate}
\usepackage{upref}
\usepackage{geometry}
\usepackage[displaymath]{lineno}
\usepackage{float}
\usepackage{dsfont}
\usepackage[colorlinks,
            linkcolor=blue,
            anchorcolor=green,
            citecolor=blue
            ]{hyperref}

\numberwithin{equation}{section}

\def\e{\varepsilon}
\def\epsilon{\varepsilon}
\def\eps{\varepsilon}

\newcommand{\ol}{\overline}
\newcommand{\wt}{\widetilde}
\newcommand{\wh}{\widehat}
\def\alb#1\ale{\begin{align*}#1\end{align*}}
\newcommand{\eqb}{\begin{equation}}
\newcommand{\eqe}{\end{equation}}

\DeclareMathOperator{\var}{var}

\newcommand{\bbE}{\mathbb{E}}
\newcommand{\bbH}{\mathbb{H}}

\newcommand{\bbR}{\mathbb{R}}
\newcommand{\bbC}{\mathbb{C}}
\newcommand{\bbP}{\mathbb{P}}

\newcommand{\cC}{\mathcal{C}}
\newcommand{\cD}{\mathcal{D}}

 \newcommand{\bfa}{\mathbf{a}}
 \newcommand{\rmL}{\mathrm{L}}
 \newcommand{\rmR}{\mathrm{R}}
\newcommand{\QT}{\mathrm{QT}}
\newcommand{\LF}{\mathrm{LF}}
\newcommand{\SLE}{\mathrm{SLE}}
\newcommand{\Wd}{\mathrm{Weld}}
\newcommand{\Md}{{\mathcal{M}}^\mathrm{disk}}

\newtheorem{theorem}{Theorem}[section]
\newtheorem{lemma}[theorem]{Lemma}
\newtheorem{proposition}[theorem]{Proposition}
\newtheorem*{proposition*}{Proposition}
\newtheorem{corollary}[theorem]{Corollary}
\newtheorem*{corollary*}{Corollary}
\newtheorem{definition}[theorem]{Definition}
\newtheorem*{definitions*}{Definitions}

\newtheorem*{example*}{\bf Example}
\theoremstyle{remark}

\numberwithin{equation}{section}
\title{The $p$-$\theta$ relation in mating of trees  {and its application to skew Brownian permutons}}
\author{Morris Ang\thanks{UC San Diego}  \qquad Xin Sun\thanks{Beijing International Center for Mathematical Research, Peking University.}\qquad Pu Yu\thanks{New York University} }
\date{}
\begin{document}

\maketitle

\begin{abstract}
In the mating-of-trees approach to Schramm-Loewner evolution (SLE) and Liouville quantum gravity (LQG), it is natural to consider two pairs of correlated Brownian motions coupled together. This arises in the scaling limit of bipolar-orientation-decorated planar maps (Gwynne-Holden-Sun, 2016) and in the related skew Brownian permuton studied by Borga et al. There are two parameters that can be used to index the coupling between the two pairs of Brownian motions, denoted as $p$ and $\theta$ in the literature: $p$ describes the Brownian motions, whereas $\theta$ describes the SLE curves on LQG surfaces.  In this paper, we derive an  exact relation between the two parameters and demonstrate its application to computing statistics of the skew Brownian permuton. Our derivation relies on the synergy between mating-of-trees and Liouville conformal field theory (LCFT), where the boundary  {three-point function}  in LCFT provides the exact solvable inputs.
\end{abstract}

\section{Introduction}

Liouville quantum gravity (LQG) coupled with Schramm-Loewner evolution (SLE) is a central topic in random conformal geometry. The mating-of-trees approach~\cite{DMS14,GHS19}  provides a powerful framework to study this coupling and its relation to decorated random planar maps. One important class of models are the so-called bipolar-orientation decorated maps introduced in~\cite{KMSW19} and further studied in~\cite{GHS16}. A closely related model called Schynder-wood-decorated triangulations was then studied in~\cite{LSW17}. Compared with previously studied planar map models, a novel feature of these models is the presence of multiple pairs of random trees instead of a single pair.  They  gained much interest in recent years  due to 
their connection to the scaling limit of  a natural class of random permutations, namely the Baxter permutations and its relatives~\cite{BM20,borga2021skewperm,BHSY23}. 

In this paper, we solve a basic question in this line of research, which is the exact relation between the $p$ and $\theta$ parameters as introduced in~\cite{GHS16}.  Here $\theta$ is the angle of an imaginary geometry flow line relative to a space-filling SLE curve, and $p$ is the proportion of time that the space-filling SLE curve is to the left of the flow line. Each parameter can be used to index the coupling between two pairs of trees and they have a 1-to-1 correspondence. We show 
(in Theorem~\ref{thm:main}) that they are related by:
\begin{equation}\label{eq:simple-rel}
   \frac{p}{1-p} =  \frac{\sin((\frac{1}{2}-\frac{\gamma^2}{8})(\pi-2\theta))}{\sin((\frac{1}{2}-\frac{\gamma^2}{8})(\pi+2\theta))} 
\end{equation}where $\gamma\in (0,2)$ is the LQG parameter. See Section~\ref{subsec:main} for the precise definitions of $p$ and $\theta$, and see Figure~\ref{fig:MOT} for the setup. Previously the relation~\eqref{eq:simple-rel} was only known when $\theta=0$ and $p=\frac12$ by symmetry, and when $\theta=\frac{\pi}{6}$, $p=\frac{1}{1+\sqrt{2}}$ and $\gamma=1$ by the relation to Schnyder-wood-decorated triangulation~\cite{LSW17}. There is another natural quantity relating the local time of a Brownian motion to the LQG length measure, which we determine in Theorem~\ref{thm:main-1}. 

Besides its intrinsic interest, the exact relation~\eqref{eq:simple-rel} between $p$ and $\theta$ has important implications in the study of the aforementioned discrete models. The key to the mating-of-trees framework is to encode the LQG/SLE coupling via Brownian motions. This reduces  {many} scaling limit questions to the convergence of random walk to Brownian motion.  {The parameter  $p$ is related to
the Brownian motions,}  hence can be used to parametrize the limiting object. For example, the family of skew Brownian permutons~\cite{borga2021skewperm},  a generalization of the Baxter permuton (i.e.\ the scaling limit of uniformly random Baxter permutations), are parametrized by $p$.   {Furthermore, it is also shown in the recent work~\cite{borga2024reconstructing} that one can reconstruct the LQG/SLE pair via a skew Brownian permuton.} On the other hand, the $\theta$ parameter is more natural in the original LQG/SLE setting. Hence knowing the relation~\eqref{eq:simple-rel} allows one to analyze the continuum object via LQG/SLE tools beyond mating-of-trees. In particular, the relation~\eqref{eq:simple-rel} yields the explicit involution rate and density function for the skew Brownian permuton (Theorem~\ref{thm-involution}). Previously, this was only known for the Baxter permuton~\cite{BHSY23} and 
the few
cases for which~\eqref{eq:simple-rel} is known.   {Moreover, our formula~\eqref{eq:simple-rel} is also closely related to the directed paths and distances in random planar maps~\cite{borga2025directed}.} See Section~\ref{subsec:Baxter} for details.

Our derivation of~\eqref{eq:simple-rel} goes
beyond the mating-of-trees framework. It relies on the exact inputs coming from the integrability of Liouville conformal field theory (LCFT), which   {can be applied to describe}  the variants of Gaussian free field naturally appearing in the mating-of-trees framework.
 In particular, we use the boundary three-point structure constant for LCFT studied in~\cite{arsz-structure-constants}. 
 We give an overview of our proof in Section~\ref{subsec:strategy}.
The earlier derivation of another basic mating-of-trees quantity, the variance of the Brownian motion~\cite{AGS21}, also relied on the synergy  with LCFT. 
 See, e.g., \cite{AHS21,nolin2023backbone,ASYZ24,LSYZ24} for other exact results in random conformal geometry obtained by combining the mating-of-trees framework and LCFT.

The rest of the paper is organized as follows. In Section~\ref{subsec:main} we state our main result. In Section~\ref{subsec:Baxter}, we explain applications of our results to the skew Brownian permuton. The proof strategy is explained in Section~\ref{subsec:strategy}. In Section~\ref{sec:pre} we provide preliminaries on SLE and LQG. In Section~\ref{sec:conf}, we prove a conformal welding result, based on which we derive the exact formula~\eqref{eq:simple-rel} in Section~\ref{sec:p-lcft}.

\subsection{Main Result}\label{subsec:main}

We briefly recall the mating-of-trees framework in~\cite{DMS14}. Let $\gamma\in(0,2)$ and $\kappa=\gamma^2$. Let $(\bbC,h,0,\infty)$ be a $\gamma$-quantum cone, $\wh h$ a whole plane Gaussian free field (GFF) independent from $h$, and $\eta'$ the north-going space-filling counterflowline of $\wh h$ from $\infty$ to $\infty$ (so $\eta'$ is a whole-plane space-filling $\SLE_{16/\kappa}$ loop rooted at $\infty$).  {Note that for any $a<b$, for $16/\kappa\geq8$, $\eta'([a,b])$ is simply connected, while for $16/\kappa\in(4,8)$,  $\eta'([a,b])$ is connected but may have pinch points.} See Sections~\ref{sec:pre-gff} and~\ref{sec:pre-ig} for the precise definitions of these objects. We parameterize $\eta'$ such that $\eta'(0)=0$ and for $-\infty<s<t<\infty$, $t-s$ is the $\gamma$-LQG area of $\eta'([s,t])$ induced by $h$. Then one can define a process $(L_t,R_t)_{t\in\bbR}$ tracing the change of the LQG-lengths of the left and right boundaries of $\eta'((-\infty,t])$. The process $(L_t,R_t)_{t\in\bbR}$ evolves as a correlated 2D Brownian motion, and determines the pair $(h,\eta')$ up to an affine transform.

\begin{figure}
    \centering
    \includegraphics[scale=0.68]{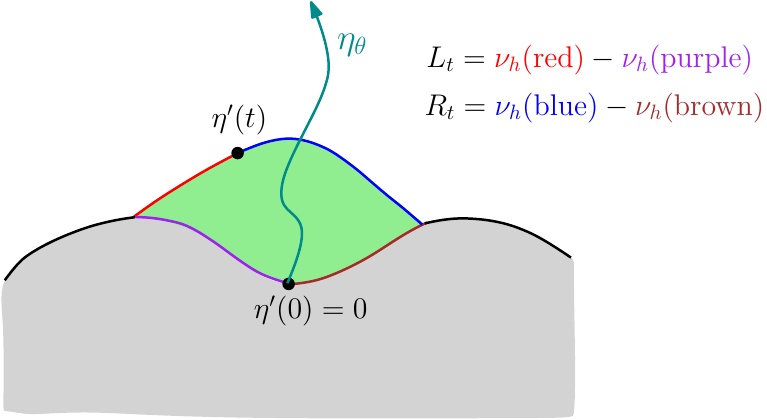}
    \caption{An illustration of the mating-of-trees setup.  {The green part is $\eta'([0,t])$.} Here $p_\gamma(\theta)$ is the same as the probability that $\eta'(t)$ is on the left of the flow line $\eta_\theta$, and the local time of $|X|$ at 0 is $2c_\gamma(\theta)$ times the quantum length of $\eta'([0,t])\cap\eta_\theta$.}
    \label{fig:MOT}
\end{figure}

Now let $\theta\in[-\frac{\pi}{2},\frac{\pi}{2}]$ and let $\eta_\theta$ be the $\theta$-angle flow line of $\wh h$ started from 0 as defined through imaginary geometry~\cite{MS17}. See Figure~\ref{fig:MOT} for the setup. Using the scale invariance of the quantum cone~\cite[Proposition 4.13]{DMS14}, the probability that $\eta'(t)$ lies on the left hand side of $\eta_\theta$ does not depend on $t$. We denote this probability by $p_\gamma(\theta)$. 

\begin{theorem}\label{thm:main}
    Let $\gamma\in(0,2)$ and   $\theta\in[-\frac{\pi}{2},\frac{\pi}{2}]$, and let $p_\gamma(\theta)$ be the constant defined above. Then
     \begin{equation}\label{eq:p-theta}
        \frac{p_\gamma(\theta)}{1-p_\gamma(\theta)} =  \frac{\sin((\frac{1}{2}-\frac{\gamma^2}{8})(\pi-2\theta))}{\sin((\frac{1}{2}-\frac{\gamma^2}{8})(\pi+2\theta))}.   
    \end{equation}
\end{theorem}

As shown in~\cite[Proposition 3.2]{GHS16}, one can define a continuous process $(X_t)_{t\geq0}$, which evolves as $(L_t)_{t\geq0}$ whenever $\eta'$ is on the right hand side of $\eta_\theta$, and evolves as $(R_t)_{t\geq0}$ whenever $\eta'$ is on the left hand side of $\eta_\theta$. This process $(X_t)_{t\geq0}$ has the law of skew Brownian motion with parameter $p=p_\gamma(\theta)$, a permuton that can be defined in terms of solutions to SDEs; see Section~\ref{subsec:Baxter} below,  {and see~\cite{Lej06} for an introduction to the skew Brownian motion}. 
Moreover, as shown in~\cite[Proposition 3.3]{GHS16}, there exists a constant $c = c_\gamma(\theta)$ such that the local time of $|X|$ at 0 is equal to $2c_\gamma(\theta)$ times the quantum length of $\eta'([0,t])\cap\eta_\theta$.  See Section~\ref{sec:p-contour} for a concrete description of the setup.  
\begin{theorem}\label{thm:main-1}
     Let $\gamma\in(0,2)$ and  $\theta\in[-\frac{\pi}{2},\frac{\pi}{2}]$, and let $c_\gamma(\theta)$ be the constant introduced above. Then 
     \begin{equation}\label{eq:c-theta}
         c_\gamma(\theta) = \frac{\cos((1-\frac{\gamma^2}{4})\theta)}{\sin(\frac{\pi\gamma^2}{8})}.
     \end{equation}
\end{theorem}
It was shown in~\cite[Corollary 5.6]{GHS16}
that $c_\gamma (0)=2$ for $\gamma=\sqrt{4/3}$. In Section~\ref{sec:p-lcft}, we derive~\eqref{eq:c-theta} for $(\theta,\gamma)\neq (0,\sqrt{4/3})$.

\subsection{Overview of related works and applications to skew Brownian permuton}\label{subsec:Baxter}

 {The objects considered in Section~\ref{subsec:main} arise from natural scaling limits of random planar maps  decorated with a statistical physics models. As studied in~\cite{GHS16}, for the bipolar orientations introduced in~\cite{KMSW19}, one can also consider their dual maps, which further generates a pair of space-filling paths on the planar maps and a lattice walk $(\mathcal{L},\mathcal{R})$ as explained in~\cite[Section 1.3]{GHS16}. One can further define a walk $\mathcal{X}$ using $(\mathcal{L},\mathcal{R})$ as in~\cite[Section 2.1]{GHS16}. The main result of~\cite{GHS16} shows that $(\mathcal{L},\mathcal{R},\mathcal{X})$ converges to the processes $(L,R,X)$ in Section~\ref{subsec:main} under the Brownian scaling. This is one example of the \emph{convergence in peanosphere sense} of random planar maps to Liouville quantum gravity, which we will briefly discuss later in this subsection. It is expected that the pair of space-filling paths on the planar maps converges to the space-filling SLE$_{16/\kappa}$ curve $\eta'$ and the space-filling SLE$_{16/\kappa}$ curve $\eta_\theta'$, where the latter is  constructed using the $\theta$-angle flow lines $\eta_\theta$ via imaginary geometry~\cite[Section 1.2.3]{ig4}. The work~\cite{GHS16} is mostly on bipolar-oriented triangulations and the parameters $(\gamma,16/\kappa,\theta) = (\sqrt{4/3},12,0)$, and one can also consider general bipolar-oriented planar maps where the  parameter $\gamma$ ranges between  $(0,\sqrt2]$.}

 {There are various ways that one can phrase the convergence of random planar maps to Liouville quantum gravity. The first example is the peanosphere convergence as above, where the random walks encoding many random planar maps  decorated with a statistical physics models converge to the Brownian motions generating SLE/LQG via the mating-of-trees framework. Examples in this direction include~\cite{She16} (tree-weighted maps and FK models), ~\cite{KMSW19,GHS16} (bipolar-orientated maps),~\cite{LSW17} (Schnyder woods) and~\cite{BHS18} (percolations on uniform triangulations). A second way is through discrete conformal embeddings, where the counting measure on vertices of embedded planar maps converges to the LQG area measure and the embedded  statistical mechanics model converges to SLE. Examples include the Smith and Tutte embedding for the mated-CRT maps~\cite{bertacco2025scaling,GMS21tutte} and the Cardy embedding~\cite{HS19} for  uniform triangulations. In forthcoming works with Holden, the third named author will extend the aforementioned embedding convergence to circle packing and Riemann uniformization embedding based on~\cite{HY26}. A third way is to view  planar maps as metric spaces equipped with graph distance. Le Gall and Miermont~\cite{LeGall13,Mie13} proved that uniform quadrangulations (equipped with a rescaling of the graph distance) converge in law with respect to the Gromov-Hausdorff
topology to a continuum metric space called \emph{the Brownian map}, which was later identified with $\sqrt{8/3}$-LQG~\cite{MSBrownian1,MSBrownian2,MSBrownian3}. Metric scaling limit results for $\gamma\neq\sqrt{8/3}$ remain open.}

 {One can also study the scaling limits of permutations, where the limit is called \emph{permutons.}} A permuton is a Borel probability measure $\mu$ on $[0,1]^2$ whose marginals are both uniform, i.e., for any $0\leq a<b\leq 1$ we have $\mu([a,b]\times[0,1])=\mu([0,1]\times[a,b])=b-a$. Any permutation $\sigma$ of $\{1,2,...,n\}$ can be viewed as a permuton by setting $\mu$ as $n$ times the Lebesgue measure on $\{[\frac{i-1}{n},\frac{i}{n}]\times [\frac{\sigma(i)-1}{n},\frac{\sigma(i)}{n}]:i=1,...,n\}$. 	 {In the setting of Baxter permutations~\cite{Bax64}, it is shown in~\cite[Theorem 1.9]{BM20} that above measure induced by Baxter permutations converges weakly to the Baxter permuton. In the model of bipolar orientations as discussed in the beginning of this subsection, the pair of space-filling paths on the planar maps also defines a permutation on edges of the map, and as in~\cite[Section 2]{BM20}, one can construct bijections between the set of bipolar-oriented planar maps, lattice walks and permutations. It is shown in~\cite[Theorem 1.17]{borga2021skewperm} that  the permuton encoded by SLE/LQG via this bijection in the scaling limit is    {the skew Brownian permuton}, which we will introduce below.} 

Permuton limits have also been investigated for various models of random permutations. For many models, the permuton limits are deterministic, for instance,  Erd\"{o}s-Szekeres permutations \cite{MR2266895}, Mallows permutations \cite{starr2009thermodynamic}, random sorting networks \cite{dauvergne2018archimedean}, almost square permutations \cite{MR4149526,borga2021almost}, and permutations sorted with the \emph{runsort} algorithm \cite{alon2021runsort}. 	For random  permutations which have a scaling limit, the limiting permutons appear to be random in many cases. In~\cite{borga2021skewperm}, Borga introduced a two-parameter family of permutons  {via the skew Brownian motion}, called the \emph{skew Brownian permuton}, which covers many of the known examples, including the Baxter permuton and the Brownian separable permuton.  {The recent works~\cite{BM22D,borga2024high} also considered permutons in high dimensions, where the limiting permutions are also random. See~\cite{borga2021random} for a full overview on random permutations and permutons.}

The skew Brownian permuton $\mu_{\rho,q}$ is indexed by parameters $\rho\in(-1,1]$ and $q\in [0,1]$, and $\mu_{-1/2,1/2}$ coincides with the Baxter permuton~\cite{BM20}. We now  recall the construction of the skew Brownian permuton for $\rho\in(-1,1)$ and $q\in [0,1]$. For  $\rho\in(-1,1)$, let $( W_{\rho}(t))_{t\in \mathbb{R}_{\geq 0}}=( X_{\rho}(t), Y_{\rho}(t))_{t\in \mathbb{R}_{\geq 0}}$ be  
	a \emph{two-dimensional Brownian motion of correlation} $\rho$. This is a continuous two-dimensional Gaussian process such that the components $ X_{\rho}$ and $ Y_{\rho} $ are standard one-dimensional Brownian motions, and $\mathrm{Cov}( X_{\rho}(t), Y_{\rho}(s)) = \rho \cdot \min\{t,s\}$. Let  $( E_{\rho}(t))_{t\in [0,1]}$ be a \emph{two-dimensional Brownian loop of correlation} $\rho$. Namely, it is a two-dimensional Brownian motion of correlation $\rho$ conditioned to stay in the non-negative quadrant $\mathbb{R}_{+}^2$ and to end at the origin at time 1, i.e.\ $ E_{\rho}(1)=(0,0)$. 
	For $q\in [0,1]$,  consider the solutions of the following family of stochastic differential equations (SDEs) indexed by $u\in [0,1]$ and driven by $ E_{\rho} = ( X_{\rho}, Y_{\rho})$:
	
	\begin{equation}\label{eq:flow_SDE_gen}
		\begin{cases}
			d Z_{\rho,q}^{(u)}(t) = \mathds{1}_{\{ Z_{\rho,q}^{(u)}(t)\geq 0\}} d Y_{\rho}(t) - \mathds{1}_{\{ Z_{\rho,q}^{(u)}(t)< 0\}} d  X_{\rho}(t)+(2q-1)\cdot d L^{ Z_{\rho,q}^{(u)}}(t),& t\in(u,1),\\
			Z_{\rho,q}^{(u)}(t)=0,&  t\in[0,u],
		\end{cases} 
	\end{equation}
	where $ L^{ Z_{\rho,q}^{(u)}}(t)$ is the symmetric local-time process at zero of $ Z_{\rho,q}^{(u)}$, i.e., 
	\begin{equation*}
		L^{ Z^{(u)}_{\rho,q}}(t)=\lim_{\varepsilon\to 0}\frac{1}{2\varepsilon}\int_0^t\mathds{1}_{\left\{ Z^{(u)}_{\rho,q}(s)\in[-\varepsilon,\varepsilon]\right\}}ds.
	\end{equation*}
   {More precisely, we say $\left\{Z^{(u)}_{\rho,q}\right\}_{u\in[0,1]}$ is a solution to~\eqref{eq:flow_SDE_gen} if, for a.e.\ $\omega$ in the probability space, for a.e.\ $u \in [0,1]$ the process $ Z^{(u)}_{\rho,q}(\omega)$ is a solution to~\eqref{eq:flow_SDE_gen}.}
   The solutions to  the SDEs~\eqref{eq:flow_SDE_gen}  exist and are unique thanks to \cite[Theorem 1.7]{borga2021skewperm}; for further details see the discussion below \cite[Theorem 1.7]{borga2021skewperm}.
    The collection of stochastic processes $\left\{ Z^{(u)}_{\rho,q}\right\}_{u\in[0,1]}$ is called the \emph{continuous coalescent-walk process} driven by $( E_{\rho},q)$. Let
	
	\begin{equation*}
		\varphi_{ Z_{\rho,q}}(t)=
		\mathrm{Leb}\left( \big\{x\in[0,t)\,|\, Z_{\rho,q}^{(x)}(t)<0\big\} \cup \big\{x\in[t,1]\,|\, Z_{\rho,q}^{(t)}(x)\geq0\big\} \right), \quad t\in[0,1].
	\end{equation*}  
	
	\begin{definition}\label{def:baxter-Z}
		Fix $\rho\in(-1,1)$ and $q\in[0,1]$. The \emph{skew Brownian permuton} of parameters $\rho, q$, denoted $ \mu_{\rho,q}$, is the push-forward of the Lebesgue measure on $[0,1]$ via the mapping $(\mathbb{I},\varphi_{ Z_{\rho,q}})$. That is,
		\begin{equation*}
			\mu_{\rho,q}(\cdot)=(\mathbb{I},\varphi_{ Z_{\rho,q}})_{*}\mathrm{Leb} (\cdot)= \mathrm{Leb}\left(\{t\in[0,1]\,|\,(t,\varphi_{ Z_{\rho,q}}(t))\in \cdot \,\}\right).
		\end{equation*} 
	\end{definition}

      An SLE/LQG description for the skew Brownian permuton $\mu_{\rho,q}$ is  given  {in~\cite[Theorem 1.17]{borga2021skewperm}.} Let $\gamma\in(0,2)$ and $\theta\in(-\frac{\pi}{2},\frac{\pi}{2})$ be chosen such that $\rho = -\cos(\frac{\pi\gamma^2}{4})$ and $q = p_\gamma(\theta)$ \footnote{ {The parameter $q$ in the skew Brownian permuton is defined in the context of the $\gamma$-quantum sphere whereas the probability $p_\gamma(\theta)$ is defined in the setting of $\gamma$-quantum cone as in~\cite{GHS16,LSW17}. It is shown in~\cite[Theorem 1.17]{borga2021skewperm} that $q=p_\gamma(\theta)$.}} where $p_\gamma(\theta)$ is the function in Theorem~\ref{thm:main}. Consider a unit area $\gamma$-LQG sphere $(\bbC,h,0,\infty)$,  {which was introduced in~\cite[Definition 4.21]{DMS14},} and an independent whole plane GFF $\widehat{h}$, and let $\eta',\eta'_{\theta-\frac{\pi}{2}}$ be space-filling counterflowlines of $\widehat{h}$ of angle 0 and $\theta-\frac{\pi}{2}$, both parameterized by the $\gamma$-LQG area induced by $h$. Then, roughly speaking, $\mu_{\rho,q}$  describes the order in which the points $\{\eta'(t):t\in[0,1]\}$  are hit by $\eta'_{\theta-\frac{\pi}{2}}$, which defines a ``permutation" of $[0,1]$.

   For a permutation $\sigma$ of $\{1,...,n\}$, we define its inversion rate by 
   $$\widetilde{\mathrm{occ}}(21, \sigma) = \frac{\#\{1\leq i_1<i_2\leq n: \sigma(i_1)>\sigma(i_2)\}}{\tbinom{n}{2}}.$$
   For a permuton $\mu$, its inversion rate is defined by
    \begin{equation}\label{eqn-permuton-21}
			\widetilde{\mathrm{occ}}(21, \mu) = 2\iint_{[0,1]^{2}}\mathds{1}_{\left\{x_1<x_2;\,y_1>y_2\right\}}\mu(dx_1dy_1)\mu(dx_2dy_2).
		\end{equation}
      It is shown in \cite[Theorem 2.5]{bassino2017universal} that if $( \sigma_n)_n$ is a sequence of random permutations converging in distribution in the permuton sense to a limiting random permuton $ \mu$, i.e.\ $ \mu_{ \sigma_n}\xrightarrow[]{d} \mu$, then 
	$ \widetilde{\mathrm{occ}}(21,\sigma_n) $ converges in distribution to $\widetilde{\mathrm{occ}}(21,\mu)$. In~\cite[Proposition 1.14]{BHSY23}, we expressed the expected inversion rate $\widetilde{\mathrm{occ}}(21, \mu_{\rho,q})$ of the skew Brownian permuton in terms of the angle $\theta$ based on the SLE/LQG description of the skew Brownian permuton. Combined with Theorem~\ref{thm:main}, we have
    \begin{corollary}\label{thm-involution}
        	For all $(\rho,q)\in(-1,1)\times[0,1]$, let $\theta\in[-\frac{\pi}{2}, \frac{\pi}{2}]$ be related to $q$ by the relation $q = p_\gamma(\theta)$ given in Theorem~\ref{thm:main} with $\gamma\in(0,2)$ such that $\rho=-\cos(\pi\gamma^2/4)$. Then
		\begin{equation}\label{eq:inversion}
			\bbE[(\widetilde{\mathrm{occ}}(21, \mu_{\rho,q}))] = \frac{\pi-2\theta}{2\pi}.
		\end{equation}
    \end{corollary}
    In~\cite[Theorem 3.8]{BHSY23}, the intensity measure $\bbE[\mu_{\rho,q}]$ is expressed in terms of the quantum area of quantum disks with fixed boundary lengths, where the weights of the quantum disks are determined linearly according to the angle $\theta$. On the other hand, in~\cite{arsz-structure-constants}, with Remy and Zhu the first two named authors computed the Laplace transform of the density function for the quantum area of quantum disks (see also~\eqref{eq-R-thin}). Therefore the intensity measure $\bbE[\mu_{\rho,q}]$ can in principle be solved by further applying the $q=p_\gamma(\theta)$ relation as in Theorem~\ref{thm:main}, but the formula is rather involved in general.

     {Finally, the recent work~\cite{borga2025directed} studied directed distances for bipolar-oriented planar maps, and conjectured that their scaling limit is the directed LQG metric,  in the same way that undirected distances in random planar
maps are discrete analogs of the undirected LQG metric~\cite{DDDF,GMMetric}. The directed LQG metric is constructed precisely in terms of the space-filling curve $\eta'$ and the angle $\theta$ space-filling $\SLE_{16/\kappa}$ curve $\eta_\theta'$ generated via imaginary geometry. See~\cite[Section 2.3]{borga2025directed} for further details.}


\subsection{Proof strategy}\label{subsec:strategy}

\begin{figure}
    \centering
    \includegraphics[scale=0.6]{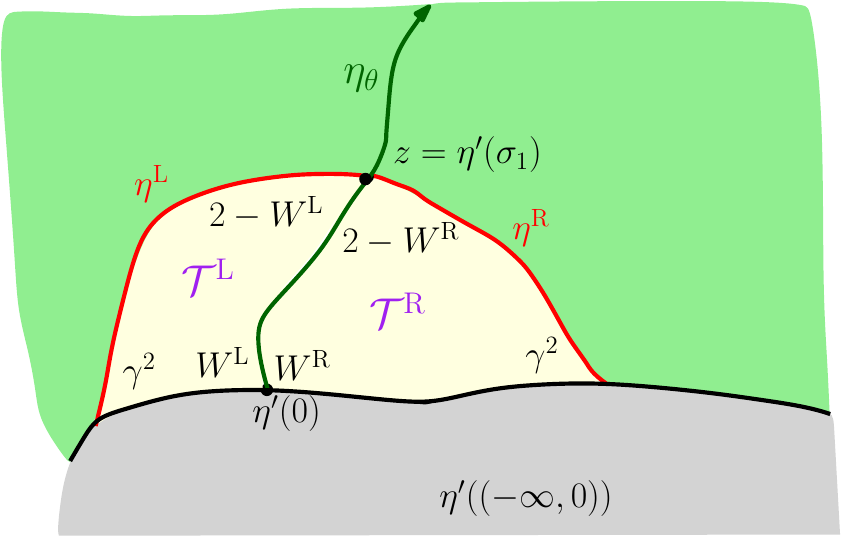}
    \caption{ {An illustration of the objects in the proof strategy. The yellow parts are $\eta'([0,\sigma_1])$ and the green part is $\eta'([\sigma_1,\infty))$. We will prove that the two yellow parts $\mathcal{T}^\rmL$ and $\mathcal{T}^\rmR$ are quantum triangles with the corresponding weights and relate the areas of $\mathcal{T}^\rmL$ and $\mathcal{T}^\rmR$ with the parameters $(p_\gamma(\theta),q_\gamma(\theta))$, and compute  $(p_\gamma(\theta),q_\gamma(\theta))$ by calculating the areas of $\mathcal{T}^\rmL$ and $\mathcal{T}^\rmR$.}}
    \label{fig:(2)}
\end{figure}

 As pioneered by Sheffield~\cite{She16a}, a cornerstone of random planar geometry is the \emph{conformal welding} of random surfaces, where the interface under the conformal welding of two independent LQG surfaces is an SLE curve. Similar type of results were also proved in~\cite{DMS14,AHS20,ASY22,AHSY23, AG23a}.

 Liouville conformal field theory (LCFT) is a 2D quantum field theory rigorously developed in~\cite{DKRV16} and subsequent works. It has been demonstrated that many natural LQG surfaces can be described by LCFT~\cite{AHS17,cercle2021unit,AHS21,ASY22}.  {LCFT is in the framework of Belavin, Polyakov, and Zamolodchikov's conformal field theory~\cite{BPZ84}. It has been extensively explored in the physics literature~\cite{DO94, ZZ96, PT_boundary_3pt} and mathematically in~\cite{KRV20,RZ20b,AGS21, GKRV20_bootstrap, GKRV21_Segal, arsz-structure-constants}, and enjoys rich and deep exact solvability.} Alongside the conformal welding of LQG surfaces mentioned earlier, with various coauthors we derived several exact formulae regarding SLE and their loop variants called conformal loop ensembles~\cite{AHS21,AGS21,AS21,ASYZ24,LSYZ24}. 

 Our proofs of Theorem~\ref{thm:main} and Theorem~\ref{thm:main-1} are another application of the theory of imaginary geometry, the conformal welding of random surfaces together with integrability of LCFT.  
 In Theorem~\ref{thm:weld-0}, we show that by drawing an independent $\SLE_\kappa(W-2;-W,W-2)$ curve on top of a quantum wedge of weight $W$, the surface that is cut off from the wedge is a \emph{quantum triangle}, a  {finite area} LQG surface with three marked boundary points we introduced in~\cite{ASY22}. Now by the conformal welding of quantum wedges~\cite[Theorem 1.2]{DMS14}, the surfaces $\mathcal{W}^\rmL$ and $\mathcal{W}^\rmR$ formed by the parts of the space-filling curve $\eta'([0,\infty))$ that lie on the left (resp.\ right) hand side of $\eta_\theta$ are quantum wedges of weights $W^\rmL$ and $W^\rmR$, where $W^\rmL$ and $W^\rmR$ are explicit in terms of $\theta$.  {See also Figure~\ref{fig:(2)} for an illustration.} We then sample a point $z$ on the flow line $\eta_\theta$,  let  $\sigma_1$ be  the time when $\eta'$ hits $z$, and consider the left and right boundaries $\eta^\rmL$ and $\eta^\rmR$ of $\eta'((-\infty,\sigma_1))$. The theory of imaginary geometry~\cite[Theorem 1.1 and Theorem 1.4]{MS16a} yields that $\eta^\rmL$ and $\eta^\rmR$ are $\SLE_\kappa(W^\rmL-2;-W^\rmL,W^\rmL-2)$  and $\SLE_\kappa(W^\rmR-2;-W^\rmR,W^\rmR-2)$ curves in  $\mathcal{W}^\rmL$ and $\mathcal{W}^\rmR$, respectively. Therefore the surfaces $\mathcal{T}^\rmL$ and $\mathcal{T}^\rmR$ formed by the parts of $\eta'((0,\sigma_1))$ that are on the left or right hand side of $\eta_\theta$ are quantum triangles. On the other hand, using the contour decomposition described in~\cite[Proposition 3.2]{GHS16} (see also Proposition~\ref{prop:contour}), we can express the constants $p_\gamma(\theta)$ and $c_\gamma(\theta)$ in terms of quantum area of the quantum triangles $\mathcal{T}^\rmL$ and $\mathcal{T}^\rmR$. Combined with the results on the boundary three-point function for LCFT from~\cite{arsz-structure-constants} as listed in Section~\ref{sec:pre-lcft-disk}, we calculate in Section~\ref{subsec:area-lcft}  the quantum areas $A_1^\rmL$ and $A_1^\rmR$ of $\mathcal{T}^\rmL$ and $\mathcal{T}^\rmR$, which eventually leads to the formulae~\eqref{eq:p-theta} and~\eqref{eq:c-theta}.
 We point out two new features of our proofs compared to previous work  {(e.g.~\cite{AHS21,AGS21,ASYZ24,nolin2023backbone})} involving SLE and LCFT. One is having quantum wedges and quantum triangles in the same picture;  the other is the crucial usage of the boundary three-point function of LCFT with positive bulk cosmological constant as derived in~\cite{arsz-structure-constants}.  


\section{Preliminaries}\label{sec:pre}

In Section \ref{sec:pre-gff}, we start with the definition of the Gaussian free field (GFF) and review the definition of quantum surfaces. In Section \ref{sec:lcft-qt}, we review the basic settings of Liouville conformal field theory and quantum triangles. In Section \ref{sec:pre-lcft-disk}, we review the LCFT description of quantum disks in \cite[Section 2]{AHS21} and provide the counterpart for quantum wedges. In Section \ref{sec:pre-ig}, we recall the basics of SLE and imaginary geometry flow lines.   

In this paper we work with non-probability measures and extend the terminology of ordinary probability to this setting. For a finite or $\sigma$-finite  measure space $(\Omega, \mathcal{F}, M)$, we say $X$ is a random variable if $X$ is an $\mathcal{F}$-measurable function with its \textit{law} defined via the push-forward measure $M_X=X_*M$. In this case, we say $X$ is \textit{sampled} from $M_X$ and write $M_X[f]$ for $\int f(x)M_X(dx)$. \textit{Weighting} the law of $X$ by $f(X)$ corresponds to working with the measure $d\tilde{M}_X$ with Radon-Nikodym derivative $\frac{d\tilde{M}_X}{dM_X} = f$, and \textit{conditioning} on some event $E\in\mathcal{F}$ (with $0<M[E]<\infty$) refers to the probability measure $\frac{M[E\cap \cdot]}{M[E]} $  over the space $(E, \mathcal{F}_E)$ with $\mathcal{F}_E = \{A\cap E: A\in\mathcal{F}\}$. If $M$ is finite, we write $|M| = M(\Omega)$ and $M^\# = \frac{M}{|M|}$ for its normalization. 

\subsection{The Gaussian free field and quantum surfaces}\label{sec:pre-gff} 

Fix a  domain $D\subset \mathbb{C}$. We consider the \textit{Neumann GFF} (or \textit{free boundary GFF}) $h_D$ on $D$ by starting with the Hilbert space $H(D)$,  the closure of $C^\infty(D)$ (the space of smooth functions with finite Dirichlet energy modulo a global constant) with respect to the inner product $(f, g)_\nabla = \int_D(\nabla f\cdot \nabla g)dxdy$.  {Two functions are considered as equivalent if their difference is a constant.}  Then set 
\begin{equation}\label{eqn-def-gff}
h_D = \sum_{n=1}^\infty \xi_nf_n
\end{equation}
where $(\xi_n)_{n\ge 1}$ is a collection of i.i.d. standard Gaussians and $(f_n)_{n\ge 1}$ is an orthonormal basis of $H(D)$. One can show that the sum \eqref{eqn-def-gff} a.s.\ converges to a random distribution independent of the choice of the basis $(f_n)_{n\ge 1}$. We emphasize that $H(D)$ and \eqref{eqn-def-gff} are defined  {modulo an additive  global constant}, which  could be fixed in various ways. For $D = \mathcal{S}$, the horizontal strip $\mathbb{R}\times (0, \pi)$, we fix the constant by requiring that every function in $H(\mathcal{S})$ has mean value zero on $\{0\}\times [0, \pi]$, 
and we denote the corresponding laws of $h_D$ by $P_{\mathcal{S}}$. 
See \cite[Section 4.1.4]{DMS14} for more details. 

For $h_{\mathcal S}$, the covariance kernel $G_{\mathcal S}(z,w) := \mathbb E[h_{\mathcal S}(z) h_{\mathcal S}(w)]$ is given by \begin{equation}\label{eqn-gff-cov}
\begin{split}
G_{\mathcal{S}}(z,w) = -\log|e^z-e^w|-\log|e^z-e^{\bar{w}}| +2\max\{\text{Re}\,z, 0\}+2\max\{\text{Re}\,w, 0\}.
\end{split}
\end{equation}

We will also need the radial-lateral decomposition of $h_{\mathcal{S}}$. Consider the subspace $H_1(\mathcal{S})\subset H(\mathcal{S})$ (resp. $H_2(\mathcal{S})\subset H(\mathcal{S})$) of functions which are constant (resp.\ have mean zero) on $[t, t+i\pi]:=\{t\}\times (0, \pi)$ for each $t\in \bbR$. Then we have the orthogonal decomposition $H(\mathcal{S}) =  H_1(\mathcal{S})\oplus H_2(\mathcal{S})$, yielding  
\begin{equation}\label{eqn-gff-decom}
h_\mathcal{S} = h_\mathcal{S}^1+h_\mathcal{S}^2
\end{equation}
where $h_\mathcal{S}^1$ and $h_\mathcal{S}^2$ are independent. 
 Moreover, writing $h_\mathcal{S}^1(t)$ for the common value of $h_\mathcal{S}^1$ on $[t, t+i\pi]$, the process $\{h_\mathcal{S}^1(t)\}_{t\in\mathbb{R}}$ agrees in law with $\{B_{2t}\}_{t\in\mathbb{R}}$ where $\{B_t\}_{t\in\mathbb{R}}$ is the standard two-sided Brownian motion. 
 See \cite[Section 4.1.6]{DMS14} for more details. 

 Now we turn to Liouville quantum gravity and  quantum surfaces. Throughout this paper, we fix the LQG coupling constant $\gamma\in(0,2)$ and set
\begin{equation*}
Q = \frac{2}{\gamma}+\frac{\gamma}{2}, \qquad \qquad \kappa = \gamma^2.
\end{equation*}

For two tuples $(D,h,z_1, ..., z_m)$ and $(\tilde{D}, \tilde{h}, \tilde{z}_1, ..., \tilde{z}_m)$, where $D$ and $\tilde{D}$ are domains on $\mathbb{C}$ with $(z_1, ..., z_m)$ and $(\tilde{z}_1, ..., \tilde{z}_m)$ being $m$ marked points on the bulk and the boundary of $D$ and $\tilde{D}$, and $h$ (resp.\,$\tilde{h}$) a distribution on $D$ (resp.\,$\tilde{D}$), we say  
\begin{equation}\label{eqn-qs-relation}
(D, h, z_1, ..., z_m)\sim_\gamma (\tilde{D}, \tilde{h}, \tilde{z}_1, ..., \tilde{z}_m)
\end{equation}
if one can find a conformal mapping $\psi:D\to \tilde D$ such that $\psi(z_j) = \tilde z_j$ for each $j$ and $\tilde{h} = \psi\bullet_\gamma h$ where 
 { $$\psi \bullet_\gamma h := h\circ \psi^{-1}+Q\log|(\psi^{-1}) '|.$$ We define an equivalence relation via $\sim_\gamma$, and}
we call an equivalence class $(D, h, z_1, ..., z_m)/{\sim_\gamma}$ a $\gamma$-\textit{quantum surface}, while a representative of this equivalence class is called \emph{an embedding} of the surface. We define \emph{curve-decorated} quantum surfaces by considering tuples $(D,h, z_1, \dots, z_m, \eta_1, \dots, \eta_n)$ where the $\eta_i$ are curves in $\ol D$, and in the definition of $\sim_\gamma$ we further require that $\psi\circ \eta_i$ agrees with $\tilde \eta_i$ up to monotone reparametrization.

For a $\gamma$-quantum surface $(D, h, z_1, ..., z_m)$, its \textit{quantum area measure} $\mu_h$ is defined by the weak limit $\mu_h := \lim_{\eps \to 0}\epsilon^{\frac{\gamma^2}{2}}e^{\gamma h_\epsilon(z)}d^2z$, where $d^2z$ is the Lebesgue area and $h_\epsilon(z)$ is the circle average of $h$ over $\partial B(z, \epsilon)$. When $D=\mathcal{S}$, we can also define the  \textit{quantum boundary length measure} $\nu_h:=\lim_{\epsilon\to 0}\epsilon^{\frac{\gamma^2}{4}}e^{\frac{\gamma}{2} h_\epsilon(x)}dx$ where $h_\epsilon (x)$ is the average of $h$ over the semicircle $\{x+\epsilon e^{i\theta}:\theta\in(0,\pi)\}$. It has been shown in \cite{DS11, SW16, berestycki-gmc} that all these weak limits ( {convergence in probability}) are well-defined  for the GFF and its variants we are considering in this paper. Moreover, the quantum area and boundary length measures are intrinsic to the quantum surface: if $\psi$ is a conformal automorphism of $\mathcal{S}$ and $\tilde h = \psi \bullet_\gamma h$, then $\psi_* \mu_h =\mu_{\tilde h}$ and $\psi_* \nu_h = \nu_{\tilde h}$. As such, the definition of $\nu_{h}$ can be extended to general simply connected domains $D \subset \bbC$ using $\bullet_\gamma$.

We now recall the notions of quantum disk and quantum wedge  introduced in \cite[Section 4.5]{DMS14}. We start from the thick case where $W>\frac{\gamma^2}{2}$. Let $\beta = \gamma+ \frac{2-W}{\gamma}<Q$.

\begin{definition}[Thick quantum disk]\label{def:thick-disk}
    Suppose $\psi_1$ and $\psi_2$ are independent distributions on $\mathcal{S}$ such that: 
	\begin{enumerate}
		\item $\psi_1$ has the same law as {$X_{\mathrm{Re}(\cdot)}$, where $X_t$ is the random function }
		\begin{equation}
		X_t:=\left\{ \begin{array}{rcl} 
			B_{2t}-(Q-\beta) t & \mbox{for} & t\ge 0\\
			\tilde{B}_{-2t} +(Q-\beta) t  & \mbox{for} & t<0
		\end{array} 
		\right.
		\end{equation}
        and $(B_t)_{t\ge 0}$ and $(\tilde{B}_t)_{t\ge 0}$ are standard Brownian motions conditioned on  $B_{2t}-(Q-\beta)t<0$ and  $ \tilde{B}_{2t} - (Q-\beta)t<0$ for all $t>0$, 
		\item $\psi_2$ has the same law as $h_{\mathcal{S}}^2$ described in \eqref{eqn-gff-decom}.
	\end{enumerate}
   Let $\mathbf{c}$  be sampled from $\frac{\gamma}{2}e^{(\beta-Q)c}dc$ independently from $(\psi_1,\psi_2)$, and set $\psi=\psi_1+\psi_2+\mathbf{c}$. We write  $\mathcal{M}^{\textup{disk}}_2(W)$ for the infinite measure describing the law of $(\mathcal{S}, \psi, -\infty, +\infty)/{\sim_\gamma}$. We call a sample from $\mathcal{M}^{\textup{disk}}_2(W)$ a (two-pointed) \emph{quantum disk of weight $W$}.
\end{definition}

\begin{definition}[Thick quantum wedge]\label{def:thick-wedge}
    Suppose $\psi_1$ and $\psi_2$ are independent distributions on $\mathcal{S}$ such that: 
	\begin{enumerate}
		\item $\psi_1$ has the same law as {$X_{\mathrm{Re}(\cdot)}$, where $X_t$ is the random function }
		\begin{equation}
		X_t:=\left\{ \begin{array}{rcl} 
			B_{2t}-(Q-\beta) t & \mbox{for} & t\ge 0\\
			\tilde{B}_{-2t} +(Q-\beta) t & \mbox{for} & t<0
		\end{array} 
		\right.
		\end{equation}
        and $(B_t)_{t\ge 0}$ and $(\tilde{B}_t)_{t\ge 0}$ are standard Brownian motions conditioned on  $ \tilde{B}_{2t} - (Q-\beta)t<0$ for all $t>0$;  
		\item $\psi_2$ has the same law as $h_{\mathcal{S}}^2$ described in \eqref{eqn-gff-decom}.
	\end{enumerate}
	 Set $\psi=\psi_1+\psi_2$. Let  $\mathcal{M}^{\textup{wedge}}_2(W)$ be the law of $(\mathcal{S}, \psi, -\infty, +\infty)/{\sim_\gamma}$. We call a sample from $\mathcal{M}^{\textup{wedge}}_2(W)$ a (two-pointed) \emph{quantum wedge of weight $W$}.
\end{definition}

 {The thick quantum disk in Definition~\ref{def:thick-disk} has disk topology and finite LQG area, while the  quantum wedge has infinite LQG area near precisely one of its two marked points.} When $0<W<\frac{\gamma^2}{2}$, we can also define the \textit{thin} quantum disks and wedges as a concatenation of weight $\gamma^2-W$ (two-pointed) thick disks along the two marked points of these disks as in \cite[Section 2]{AHS20}. 

\begin{definition}[Thin quantum disk]\label{def-thin-disk}
	For $W\in(0, \frac{\gamma^2}{2})$, the infinite measure $\mathcal{M}_2^{\textup{disk}}(W)$ on two-pointed beaded surfaces is defined as follows. First sample $T$ from $(1-\frac{2}{\gamma^2}W)^{-2}\textup{Leb}_{\mathbb{R}_+}$, then sample a Poisson point process $\{(u, \mathcal{D}_u)\}$ from the intensity measure $\textbf{1}_{t\in [0,T]}dt\times \mathcal{M}_2^{\textup{disk}}(\gamma^2-W)$ and finally concatenate the disks $\{\mathcal{D}_u\}$ according to the ordering induced by $u$. The total sum of the left (resp. right) boundary lengths of all the $\mathcal{D}_u$'s is referred as the left (resp. right) boundary length of the thin quantum disk.
\end{definition}

\begin{definition}[Thin quantum wedge]\label{def-thin-wedge}
	For $W\in(0, \frac{\gamma^2}{2})$, the probability measure $\mathcal{M}_2^{\textup{wedge}}(W)$ on two-pointed beaded surfaces is defined as follows. Sample a Poisson point process $\{(u, \mathcal{D}_u)\}$ from the intensity measure $\textbf{1}_{t>0}dt\times \mathcal{M}_2^{\textup{disk}}(\gamma^2-W)$ and  concatenate the disks $\{\mathcal{D}_u\}$ according to the ordering induced by $u$.
\end{definition}

We give a heuristic interpretation of these definitions. Note that one can interpret the ordered collection of doubly-marked quantum disks $\{\mathcal{D}_u\}$ as if we are concatenating the surfaces $\{\mathcal{D}_u\}$ by ``gluing'' them at their marked points, as shown in Figure~\ref{fig:thindisk}. These collections of doubly-marked quantum surfaces are the \emph{beaded surfaces}.

	\begin{figure}[ht!]
		\begin{center}
			\includegraphics[width=.7\textwidth]{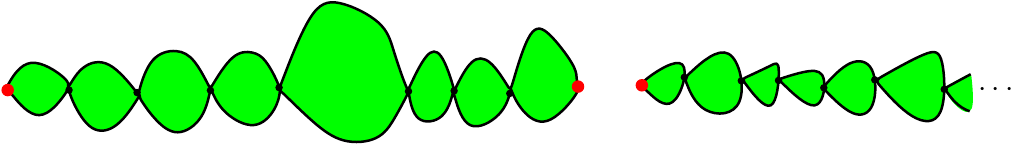}  
			\caption{Schematic representation of a sample of a thin quantum disk and thin quantum wedge of weight $W\in(0,\frac{\gamma^2}{2})$ as the concatenation of weight $\gamma^2-W$ thick quantum disks. The green bubbles correspond to the thick quantum disks $\{\mathcal{D}_u\}$ involved in the construction. Note that there are in fact (countably) infinite many thick quantum disks which are not drawn near the two endpoints (shown in red) and between each pair of macroscopic disks. }\label{fig:thindisk}
		\end{center}
	\end{figure}

We also mention the definition of quantum cones. Let $\mathcal{C}:=\bbR\times [0,2\pi]/\sim$ be the infinite cylinder where we identify $\bbR\times\{0\}$ with  $\bbR\times\{2\pi\}$ via the equivalence $x\sim x+2\pi i$. {Let $H(\cC)$ be the closure of the space of smooth functions with finite Dirichlet energy with mean zero on $[t, t + 2\pi i] := \{ t\} \times [0,2\pi]$, and let $H_1(\cC) \subset H(\cC)$ (resp.\ $H_2(\cC) \subset H(\cC)$) be the space of functions which are constant (resp.\ have mean zero) on $[t, t + 2\pi i]$ for each $t \in \bbR$. We can similarly define the GFF $h_\cC$ and its independent decomposition $h_\cC = h_\cC^1 + h_\cC^2$ as before.} 
\begin{definition}[Quantum cone]
    Fix $W>0$ and let $\alpha:=Q-\frac{W}{2\gamma}$. Suppose $\psi_1$ and $\psi_2$ be independent distributions on $\mathcal{C}$ such that: 
	\begin{enumerate}
		\item $\psi_1$ has the same law as {$X_{\mathrm{Re}(\cdot)}$, where $X_t$ is the random function }
		\begin{equation}
		X_t:=\left\{ \begin{array}{rcl} 
			B_{t}-(Q-\alpha) t & \mbox{for} & t\ge 0\\
			\tilde{B}_{-t} +(Q-\alpha) t & \mbox{for} & t<0
		\end{array} 
		\right.
		\end{equation}
        and $(B_t)_{t\ge 0}$ and $(\tilde{B}_t)_{t\ge 0}$ are standard Brownian motions conditioned on  $ \tilde{B}_{t} - (Q-\alpha)t<0$ for all $t>0$;  
		\item $\psi_2$ has the same law as $h_{\mathcal{C}}^2$.
	\end{enumerate}
	 Set $\psi=\psi_1+\psi_2$. Let  $\mathcal{M}^{\textup{cone}}_2(W)$ be the law of $(\mathcal{C}, \psi, -\infty, +\infty)/\sim_\gamma$. We call a sample from $\mathcal{M}^{\textup{cone}}_2(W)$ a (two-pointed) \emph{quantum cone of weight $W$}.
\end{definition}
We comment that the above quantum surfaces are defined using the GFF on the strip $\mathcal{S}$ and cylinder $\mathcal{C}$, and they can be embedded in other domains via  the equivalence relation $\sim_\gamma$.

\subsection{Liouville conformal field theory and quantum triangles}\label{sec:lcft-qt}
In this section, we review the fundamentals of Liouville CFT, and the definitions of quantum triangles.

We begin with the  LCFT on the strip. Recall that $P_\mathcal{S}$ is the law of the Neumann GFF on $\mathcal S$ normalized to have mean zero on $\{ 0 \} \times [0,\pi]$. For $\beta\in\bbR$, we shall use the shorthand 
\begin{equation}\label{eqn-delta-alpha}
\Delta_\beta:=\frac{\beta}{2}(Q-\frac{\beta}{2}).
\end{equation}

\begin{definition}\label{def-lf-H}
	Let $(h, \mathbf{c})$ be sampled from $P_\mathcal{S}\times [e^{-Qc}dc]$ and set $\phi = h - Q|\mathrm{Re}z|+\mathbf{c}$. We say $\phi$ is a \textup{Liouville field} on $\mathcal{S}$ and let $\textup{LF}_{\mathcal{S}}$ be its law.
\end{definition}

\begin{definition}[Liouville field with insertions]\label{def-lf-strip}
	Let $(h, \mathbf{c})$ be sampled from $P_{\mathcal{S}}\times[e^{(\frac{\beta_1+\beta_2+\beta_3}{2}-Q)c}dc]$ with $\beta_1, \beta_2, \beta_3\in\mathbb{R}$.  
	Let $\phi(z) = h(z)+\frac{\beta_1+\beta_2-2Q}{2}|\textup{Re}z|+\frac{\beta_1-\beta_2}{2}\textup{Re}z+\frac{\beta_3}{2}G_{\mathcal{S}}(z, 0)+\mathbf{c}$. We write $\textup{LF}_{\mathcal{S}}^{(\beta_1, +\infty), (\beta_2, -\infty), (\beta_3, 0)}$ for the law of $\phi$,  {where the insertion points are $\pm\infty$ and 0.}
\end{definition}

 {The insertions above, which is essentially a consequence of Girsanov's theorem, can be understood  as weighting the law of the field by $e^{\frac{\beta}{2}\phi(x)}$ and applying an appropriate normalization.} If $\beta_3=0$, then there are only two insertion points at $\pm\infty$ and we denote the corresponding measure by $\textup{LF}_{\mathcal{S}}^{(\beta_1, +\infty), (\beta_2, -\infty)}$.

Now we recall the notion of quantum triangle as in \cite{ASY22}. It is a quantum surface of  {finite area} and parameterized by weights $W_1,W_2,W_3>0$. If all the weights are greater than $\frac{\gamma^2}2$, the quantum triangle is defined via the Liouville field with three insertions, and the definition extends to the case where some of the $W_i$ lie in $(0,\frac{\gamma^2}2)$ via the thick-thin duality. See Figure~\ref{fig-qt}. 

\begin{definition}[Thick quantum triangles]\label{def-qt-thick}
	Fix {$W_1, W_2, W_3>\frac{\gamma^2}{2}$}. Set $\beta_i = \gamma+\frac{2-W_i}{\gamma}<Q$ for $i=1,2,3$, and sample $\phi$ from $\frac{1}{(Q-\beta_1)(Q-\beta_2)(Q-\beta_3)}\textup{LF}_{\mathcal{S}}^{(\beta_1, +\infty), (\beta_2, -\infty), (\beta_3, 0)}$. Then we define the infinite measure $\textup{QT}(W_1, W_2, W_3)$ to be the law of $(\mathcal{S}, \phi, +\infty, -\infty, 0)/{\sim_\gamma}$.
\end{definition}

\begin{definition}[Thin quantum triangles]\label{def-qt-thin}
	Fix {$W_1, W_2, W_3\in (0,\frac{\gamma^2}{2})\cup(\frac{\gamma^2}{2}, \infty)$}. Let $I:=\{i\in\{1,2,3\}:W_i<\frac{\gamma^2}{2}\}$ and define  $\textup{QT}(W_1, W_2, W_3)$ to be the law of the quantum surface $S$ constructed as follows. Let $\tilde{W}_i = W_i$ if $i\notin I$, and $\tilde{W}_i = \gamma^2-W_i$ if $i\in I$. First sample the surface $S_0 := (D, \phi, \tilde a_1, \tilde a_2, \tilde a_3)$ from $\textup{QT}(\tilde{W}_1, \tilde{W}_2, \tilde{W}_3)$;  {the marked point $\tilde a_i$ corresponds to the vertex of weight $\tilde W_i$.} For each $i\in I$, {independently sample a surface $S_i$ from $(1-\frac{2W_i}{\gamma^2})\mathcal{M}_2^{\textup{disk}}(W_i)$} and concatenate it with $S_0$ at point $\tilde a_i$. Let $S$ be the concatenated quantum surface with three marked points $a_1,a_2,a_3$, where $a_i = \tilde a_i$ if $i \not \in I$, and $a_i$ is the second marked point of $S_i$ if $i \in I$. 
\end{definition} 

\begin{figure}[ht]
	\centering
	\includegraphics[scale=0.43]{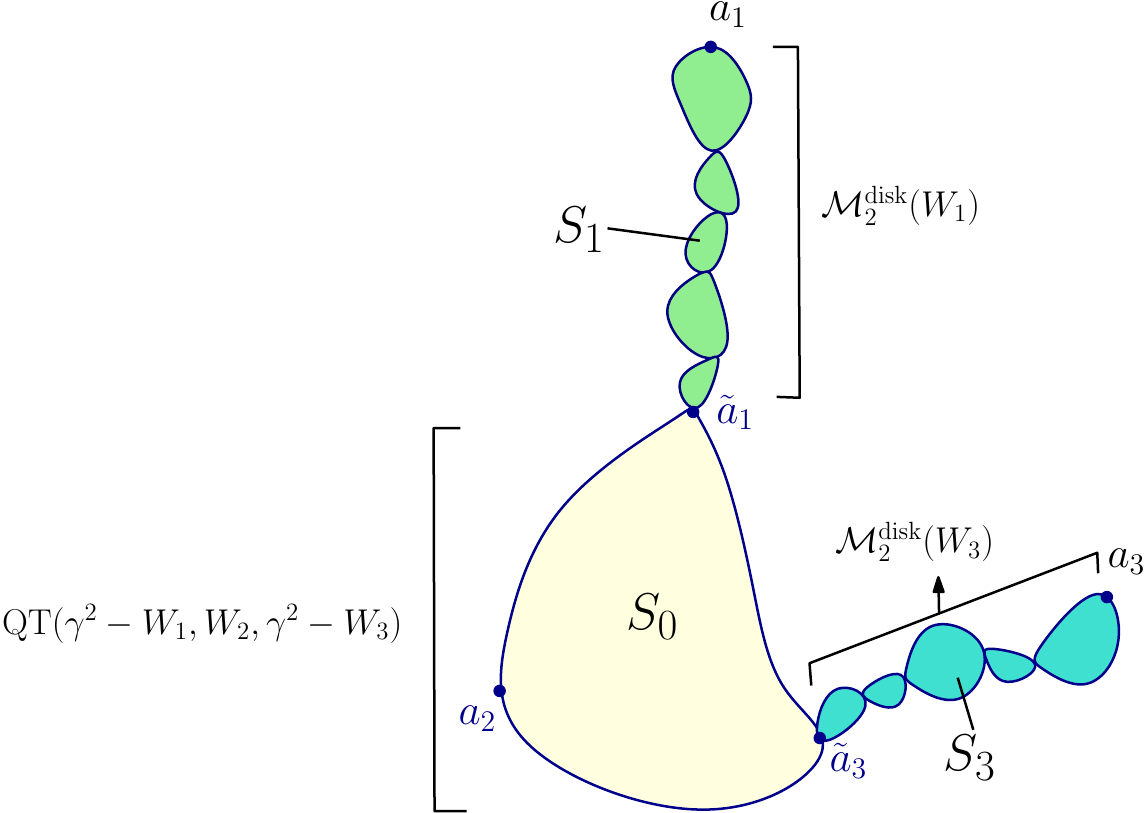}
	\caption{A quantum triangle with $W_2>\frac{\gamma^2}{2}$ and $W_1,W_3<\frac{\gamma^2}{2}$ embedded as $(D,\phi,a_1,a_2,a_3).$  The two thin disks (colored green) are concatenated with the thick triangle (colored yellow) at points $\tilde{a}_1$ and $\tilde{a}_3$.}\label{fig-qt}
\end{figure}

 {We comment that the beaded parts of these surfaces are in accordance with their Bessel process encoding as in~\cite[Section 4.4]{DMS14} where the dimension $d<2$.} If $W_i = \frac{\gamma^2}2$ for some $i$, then $\QT(W_1,W_2,W_3)$ is defined by a limiting procedure. Since this is not needed in this paper, we omit the exact definition here. See \cite[Section 2.5]{ASY22} for more details.

We briefly recall the boundary length law of the three-pointed Liouville fields.  {Throughout this paper, the terminologies ``boundary length", ``quantum length" and ``quantum boundary length" are all considered under the context of the LQG length measure $\nu_\phi$.} We start with the \emph{disintegration} of  quantum surfaces with respect to boundary length. The quantum wedges have infinite boundary lengths. For quantum disks, one has the following disintegration
$$\Md_2(W) = \int_0^\infty \Md_2(W;\ell)\ d\ell$$
where each $\Md_2(W;\ell)$ is a  measure supported on the space of two-pointed quantum surfaces with right boundary arc having quantum length $\ell$. Similarly, for quantum triangles, we have the identity 
$$\QT(W_1,W_2,W_3) = \int_0^\infty \QT(W_1,W_2,W_3;\ell)\ d\ell$$
where each $\QT(W_1,W_2,W_3;\ell)$ is a measure supported on three-pointed quantum surfaces such that the boundary arc between the weight $W_1$ and $W_2$ points has quantum length $\ell$.  One can also disintegrate over the quantum area of the surfaces.  {As shown in~\cite[Section 2.6]{AHS20}, these disintegration can be done by considering regular conditional
probability distributions as in~\cite[Chapter 6]{Kal97} and extending to the $\sigma$-finite measures $\Md_2(W)$ or $\QT(W_1,W_2,W_3)$ by exhaustion.}  {The measures $\Md_2(W;\ell)$ and $\QT(W_1,W_2,W_3;\ell)$ are defined for a.e.\ $\ell$ in the disintegration formula, and can be extended to the full range $\ell\in(0,\infty)$ in a continuous way;  {the quantum disk case has been treated in~\cite[Section 2.6 and Section 4]{AHS20}; for quantum triangles,   one can  use the same decomposition of the field as in~\cite[Lemma 2.21]{AHS20} and~\cite[Section 2.5]{ASY22} and apply the identical arguments as in the proof of~\cite[Proposition 2.23]{AHS20} using the aforementioned decomposition.}} See \cite[Section 2.6]{AHS20} for more background on this disintegration.

To describe the boundary length laws, we will need several special functions. The functions $\bar{R}$ and $\bar{H}$    {are the LCFT two-point and three-point functions respectively, where the bulk cosmological constant is set to zero, and each boundary arc has a corresponding boundary cosmological constant. These functions} are introduced for more general parameters (see \cite[Page 6-8]{RZ20b}) but for simplicity we only define the ones which will appear later;    {namely, those having only one nonzero boundary cosmological constant}.  For $b>0$, the \emph{double-gamma function} $\Gamma_b(z)$ is the meromorphic function on $\bbC$ such that for $\textup{Re}z>0$,
\begin{equation}\label{eqn-double-gamma}
\log\Gamma_b(z) = \int_0^\infty \frac{1}{t}\bigg(\frac{e^{-zt}-e^{-\frac{(b^2+1)t}{2b}}}{(1-e^{-bt})(1-e^{-\frac{1}{b}t})} -\frac{(b^2+1-2bz)^2}{8b^2}e^{-t}+\frac{2bz-b^2-1}{2bt} \bigg)dt.
\end{equation}
The function $\Gamma_b$ has poles at $\{-m b - n b^{-1} \:: \: m, n \in \mathbb Z_{\geq 0}\}$,
and satisfies the shift equations 
\begin{equation}\label{eqn-gamma-shift}
\frac{\Gamma_b(z)}{\Gamma_b(z+b)} = \frac{1}{\sqrt{2\pi}}\Gamma(bz)b^{-bz+\frac{1}{2}}, \ \ \frac{\Gamma_b(z)}{\Gamma_b(z+b^{-1})} = \frac{1}{\sqrt{2\pi}}\Gamma(b^{-1}z)b^{\frac{z}{b}-\frac{1}{2}}.
\end{equation}
For $\mu>0$, let
\begin{equation}
\bar{R}(\beta,\mu,0) = \bar{R}(\beta,0,\mu) = \mu^{\frac{2(Q-\beta)}{\gamma}}\frac{(2\pi)^{\frac{2(Q-\beta)}{\gamma}-\frac{1}{2}}(\frac{2}{\gamma})^{\frac{\gamma(Q-\beta)}{2}-\frac{1}{2}}   }{ (Q-\beta)\Gamma(1-\frac{\gamma^2}{4})^{\frac{2(Q-\beta)}{\gamma}} } \frac{\Gamma_{\frac{\gamma}{2}}(\beta-\frac{\gamma}{2})}{\Gamma_{\frac{\gamma}{2}}(Q-\beta)}.
\end{equation}
Finally set $\bar{\beta} = \beta_1+\beta_2+\beta_3$ and
\begin{equation}\label{eq-bar-H}
\bar{H}_{(0,1,0)}^{(\beta_1, \beta_2, \beta_3)} := \frac{ (2\pi)^{\frac{2Q-\bar{\beta}+\gamma}{\gamma}} (\frac{2}{\gamma})^{(\frac{\gamma}{2}-\frac{2}{\gamma})(Q-\frac{\bar{\beta}}{2})-1}   }{ \Gamma(1-\frac{\gamma^2}{4})^{\frac{2Q-\bar{\beta}}{\gamma}} \Gamma(\frac{\bar{\beta}-2Q}{\gamma}) }\frac{\Gamma_{\frac{\gamma}{2}}(\frac{\bar{\beta}}{2}-Q) \Gamma_{\frac{\gamma}{2}}(\frac{\bar{\beta}-2\beta_2}{2}) \Gamma_{\frac{\gamma}{2}}(\frac{\bar{\beta}-2\beta_1}{2} ) \Gamma_{\frac{\gamma}{2}} (Q-\frac{\bar{\beta}-2\beta_3}{2} )  }{ \Gamma_{\frac{\gamma}{2}}(Q)\Gamma_{\frac{\gamma}{2}}(Q-\beta_1)\Gamma_{\frac{\gamma}{2}}(Q-\beta_2)\Gamma_{\frac{\gamma}{2}}(\beta_3)   }.
\end{equation}
  {For the two-point function $\bar{R}(\beta,\mu, 0)$, the boundary insertions both have weight $\beta$ and the boundary cosmological constants are $\mu$ and $0$. For the three-point function $\bar{H}^{\beta_1, \beta_2, \beta_3)}_{(0,1,0)}$, the boundary insertions have weights $\beta_1, \beta_2, \beta_3$, and the subscript indicates that the boundary cosmological corresponding to the boundary arc between the insertions of weights $\beta_1$ and $\beta_2$ is set to 1 while the remaining two boundary cosmological constants are set to 0.}

In Lemma~\ref{lm:bdry-length-law}, we give the disintegration of $\LF_{\mathcal S}^{(\beta_1, +\infty), (\beta_2, -\infty), (\beta_3, 0)}$ with respect to the quantum length of a boundary arc, and identify the law of this length in terms of $\bar{H}_{(0,1,0)}^{(\beta_1, \beta_2, \beta_3)}$.    {In other words, Lemma~\ref{lm:bdry-length-law} relates the special function~\eqref{eq-bar-H} to the probabilistically-constructed LCFT three-point function. The probabilistic construction does not make sense when $\beta_3 \geq Q$ and one or more boundary cosmological constants corresponding to boundary arcs adjacent to the weight $\beta_3$ vertex are nonzero, since the boundary lengths of these boundary arcs are infinite a.s., but since these cosmological constants are set to zero we can allow $\beta_3 \in \bbR$. }

\begin{lemma}\label{lm:bdry-length-law}
    Fix $\beta_1,\beta_2<Q$ and $\beta_3\in\bbR$. Write $\bar{\beta}=\beta_1+\beta_2+\beta_3$ and suppose
    \begin{equation}\label{eqn:seiberg}
        |\beta_1-\beta_2|<\beta_3, \qquad \bar{\beta}>\gamma.
    \end{equation}
    Let $h$ be a sample from $P_\mathcal{S}$ and set $\phi = h+ \frac{\beta_1+\beta_2-2Q}{2}|\textup{Re}z|+\frac{\beta_1-\beta_2}{2}\textup{Re}z+\frac{\beta_3}{2}G_{\mathcal{S}}(z, 0)$ (i.e., a sample from $\textup{LF}_{\mathcal{S}}^{(\beta_1, +\infty), (\beta_2, -\infty), (\beta_3, 0)}$ but without the additive constant $\mathbf{c}$)
     and $L:=\nu_\phi(\bbR+\pi i)$.  For $\ell>0$, let $\textup{LF}_{\mathcal{S},\ell}^{(\beta_1, +\infty), (\beta_2, -\infty), (\beta_3, 0)}$ be the law of $\phi+\frac{2}{\gamma}\log\frac{\ell}{L}$ under the reweighted measure $\frac{2}{\gamma}\frac{\ell^{\frac{1}{\gamma}(\bar{\beta}-2Q)-1}}{ {L_{}}^{\frac{1}{\gamma}(\bar{\beta}-2Q)}}P_{\mathcal{S}}(dh)$. Then 
    \begin{equation}
      \textup{LF}_{\mathcal{S}}^{(\beta_1, +\infty), (\beta_2, -\infty), (\beta_3, 0)} = \int_0^\infty  \textup{LF}_{\mathcal{S},\ell}^{(\beta_1, +\infty), (\beta_2, -\infty), (\beta_3, 0)}d\ell;\ \ |\textup{LF}_{\mathcal{S},\ell}^{(\beta_1, +\infty), (\beta_2, -\infty), (\beta_3, 0)}|=\frac{2}{\gamma}{\bar{H}_{(0,1,0)}^{(
		\beta_1,\beta_2,\beta_3)}} \ell^{\frac{\bar{\beta}-2Q}{\gamma}-1}.
    \end{equation}
\end{lemma}
 {The reweighting  of the measure $P_{\mathcal{S}}(dh)$ as above is in accordance of the boundary length law of quantum triangles as in Proposition~\ref{prop:QT-bdry-length} below.} Lemma~\ref{lm:bdry-length-law} 
is a standard application of Fubini's theorem and the integrability results in \cite{RZ20a}; the proof is identical to, e.g., \cite[Lemma 4.2]{AHS21} and \cite[{Lemma 2.27}]{ASY22},  {where one consider any nonnegative measurable function $F$ on $H^{-1}(\mathcal{S})$ and apply a change of variable $c = \frac{2}{\gamma}\log\frac{\ell}{L_{}}$ to the integral $$\int_0^\infty\int F(\phi+\frac{2}{\gamma}\log\frac{\ell}{L_{}})\frac{2}{\gamma}\frac{\ell^{\frac{1}{\gamma}(\bar{\beta}-2Q)-1}}{L_{}^{\frac{1}{\gamma}(\bar{\beta}-2Q)}}P_{\mathcal{S}}(dh)d\ell.$$} We omit the details. 

 {Next, we state a related result from \cite{ASY22} expressing the boundary length law of quantum triangles in terms of $\bar{H}$.}

\begin{proposition}[{Proposition 2.24 of \cite{ASY22}}]\label{prop:QT-bdry-length}
    Fix {$W_1, W_2\in (0,\frac{\gamma^2}{2})\cup(\frac{\gamma^2}{2}, \infty)$} and $W_3>\frac{\gamma^2}{2}$. Let $\beta_i=\gamma+\frac{2-W_i}{\gamma}$ and $\tilde{\beta}_i = Q-|\beta_i-Q|$ for $i=1,2,3.$ Suppose 
    \begin{equation}
    |\tilde{\beta}_1-\tilde{\beta}_2|<\tilde{\beta}_3; \qquad 
    \ \tilde{\beta}_1+\tilde{\beta}_2+\tilde{\beta}_3>\gamma.
    \end{equation}
    Then for $\ell>0$, $|\QT(W_1,W_2,W_3;\ell)| = \frac{2}{\gamma(Q-\beta_1)(Q-\beta_2)(Q-\beta_3)}|\bar{H}_{(0,1,0)}^{(
	\beta_1,\beta_2,\beta_3)}| \ell^{\frac{\bar{\beta}-2Q}{\gamma}-1}$ (where $\ell$ corresponds to the quantum length of the  {boundary} arc between the weight $W_1$ and $W_2$ vertices).
\end{proposition}

Finally, we introduce the three-point structure constant of Liouville CFT with nonzero bulk cosmological constant. Let $\partial_1 \mathcal S = (0,+\infty)$, $\partial_2 \mathcal S = \bbR + \pi i$ and $\partial_3 \mathcal S = (-\infty, 0)$ be the three components of $\partial \mathcal S \backslash \{ 0, +\infty, -\infty\}$.  

\begin{definition}\label{def-H}
    Suppose $\beta_1, \beta_2, \beta_3 \in \bbR$ satisfy the Seiberg bounds $\sum_{i=1}^3 \beta_i > 2Q$ and $\beta_i < Q$ for $i = 1,2,3$, and let $\mu_1, \mu_2, \mu_3 \geq 0$. Set
    \[H^{(\beta_1, \beta_2, \beta_3)}_{(\mu_1, \mu_2, \mu_3)} := \LF_{\mathcal S}^{(\beta_1, +\infty), (\beta_2, -\infty), (\beta_3, 0)}[\exp(- \mu_\phi(\mathcal S) - \sum_{i=1}^3 \mu_i \nu_\phi(\partial_i \mathcal S)]. \]
    For $\sigma_1, \sigma_2, \sigma_3 \in [ -\frac1{2\gamma} + \frac Q2, \frac1{2\gamma} + \frac Q2]$, define
    \eqb\label{eq-H-sigma}
        H
        \begin{pmatrix}
        	\beta_1, \beta_2, \beta_3 \\
        	\sigma_1,  \sigma_2,   \sigma_3 
        \end{pmatrix} = H^{(\beta_1, \beta_2, \beta_3)}_{(\mu_B(\sigma_1), \mu_B(\sigma_2), \mu_B(\sigma_3))}, \qquad \mu_B (\sigma) := (\sin(\frac{\pi\gamma^2}4))^{-1/2} \cos(\pi \gamma(\sigma - \frac Q2)) .
    \eqe
\end{definition}
The range of $\sigma_i$ specified above is such that the  $\mu_B(\sigma_i)$ are nonnegative reals. 
	As an immediate consequence of this definition and Definition~\ref{def-qt-thick} we have the following.
	\begin{lemma} \label{lem-H-thick}
		Let $W_1, W_2, W_3 > \frac{\gamma^2}2$ and let $\beta_i = \gamma + \frac{2-W_i}{\gamma} < Q$ for $i = 1,2,3$. Suppose $\sum_i \beta_i > 2Q$. For a sample from $\QT(W_1, W_2, W_3)$, let $A$ denote its quantum area, and let $L_1, L_2, L_3$ denote the quantum lengths of its boundary arcs in counterclockwise order from its third marked point. Then 
        \eqb\label{eq-QT-H}
	\QT(W_1, W_2, W_3)[\exp(-A - \sum_{i=1}^3 \mu_i L_i)] = \prod_{i=1}^3 (Q-\beta_i)^{-1} H^{(\beta_1, \beta_2, \beta_3)}_{(\mu_1, \mu_2, \mu_3)}.
    \eqe
	\end{lemma}
Analogously to Proposition~\ref{prop:QT-bdry-length}, \cite[Theorem 1.1]{arsz-structure-constants} shows that $H$ agrees with the \emph{Ponsot-Teschner formula} conjectured in \cite{PT_boundary_3pt}. For our arguments, we will not use this exact formula, but instead manipulate various functional equations of $H$. To that end we need to  meromorphically extend $H$:

\begin{proposition}[{\cite[Theorem 1.1]{arsz-structure-constants}}] \label{prop-H-extend}
    The function $H
	\begin{pmatrix}
		\beta_1, \beta_2, \beta_3 \\
		\sigma_1,  \sigma_2,   \sigma_3
    \end{pmatrix}$ extends meromorphically to $H:\mathbb C^6 \to \mathbb C$.
\end{proposition}
We then extend the definition of $H^{(\beta_1, \beta_2, \beta_3)}_{(\mu_1, \mu_2, \mu_3)}$ to $\beta_1, \beta_2, \beta_3 \in \mathbb C$ and $\mu_1, \mu_2, \mu_3 \geq 0$ via~\eqref{eq-H-sigma}.
As we now see, the meromorphic continuation can be used to describe certain thin quantum triangles. The proof follows from the thick-thin duality defining thin quantum triangles and a reflection identity in Liouville CFT.

\begin{lemma}
	\label{lem-H-thin}
	
	Let $W_1 \in (0,\frac{\gamma^2}2)$ and
	$W_2, W_3 > \frac{\gamma^2}2$, let $\beta_i = \gamma + \frac{2-W_i}{\gamma}$ for $i =1,2,3$, and let $\mu_1, \mu_2, \mu_3 \geq 0$.   Suppose that $(2Q - \beta_1) + \beta_2 + \beta_3 > 2Q$, then~\eqref{eq-QT-H} holds. 
\end{lemma}
\begin{proof}
	Our argument depends on the \emph{reflection coefficient} $R$ defined as follows. Let $W(\beta) = \gamma(\gamma + \frac2\gamma - \beta)$.
	For $\sigma, \sigma' \in [ -\frac1{2\gamma} + \frac Q2, \frac1{2\gamma} + \frac Q2]$, following \cite[Lemma 2.16]{arsz-structure-constants} we set
	\[R(\beta, \sigma, \sigma') := \frac{2(Q-\beta)}\gamma \mathcal{M}_2^\textup{disk}(W(\beta))[e^{-A^D - \mu_B(\sigma)L^D - \mu_B(\sigma')R^D} - 1] \qquad \text{ for } \beta \in (\frac2\gamma, Q)\]
	where $A^D, L^D, R^D$ are the quantum area and  boundary arc lengths of the sampled quantum disk. We extend the definition to $\beta \in (Q, Q+\frac\gamma2)$ via $R(\beta, \sigma, \sigma') := 1/R(2Q - \beta, \sigma, \sigma')$ \cite[Equation (3.36)]{arsz-structure-constants}. Taking as input the above two definitions, the argument of \cite[Proposition 3.6]{AHS21} gives 
	\eqb
	\label{eq-R-thin}
	R(\beta, \sigma, \sigma') = \frac{2(Q-\beta)}\gamma \mathcal{M}_2^\textup{disk}(W(\beta))[e^{-A^D - \mu_B(\sigma)L^D - \mu_B(\sigma')R^D}] \qquad \text{ for } \beta \in (Q, Q+\frac\gamma2).
	\eqe
	
	We now turn to the main argument. 
	Let $\tilde W_1 = \gamma^2 - W_1$ and $\tilde \beta_1 = 2Q - \beta_1$, and let $\sigma_i \in [-\frac1{2\gamma} + \frac Q2, \frac1{2\gamma} + \frac Q2]$ satisfy $\mu_i = \mu_B(\sigma_i)$ for $i=1,2,3$. By Definition~\ref{def-qt-thin}
	\[\QT(W_1, W_2, W_3)[e^{-A - \sum_{i=1}^3 \mu_i L_i}] = (1 - \frac{2W_1}{\gamma^2}) \cdot \QT(\tilde W_1, W_2, W_3)[e^{-  A - \sum_{i=1}^3 \mu_i L_i}]\cdot \mathcal{M}_2^{\textup{disk}}(W_1)[e^{-A^D - \mu_1 L^D - \mu_2 R^D}].\]
	By Lemma~\ref{lem-H-thick} and~\eqref{eq-R-thin}, this expression equals 
	\[\frac2\gamma(\beta_1-Q) \cdot (Q-\tilde \beta_1)^{-1}(Q-\beta_2)^{-1}(Q-\beta_3)^{-1}  H
	\begin{pmatrix}
		\tilde \beta_1, \beta_2, \beta_3 \\
		\sigma_1,  \sigma_2,   \sigma_3
	\end{pmatrix} \cdot \frac\gamma{2(Q - \beta_1)}R(\beta_1, \sigma_1, \sigma_2).\]
	\cite[Equation (3.35)]{arsz-structure-constants} gives $H
	\begin{pmatrix}
		\tilde \beta_1, \beta_2, \beta_3 \\
		\sigma_1,  \sigma_2,   \sigma_3
	\end{pmatrix} R(\beta_1, \sigma_1, \sigma_2) = H
	\begin{pmatrix}
	\beta_1, \beta_2, \beta_3 \\
	\sigma_1,  \sigma_2,   \sigma_3
	\end{pmatrix}$, so this expression equals $\prod_{i=1}^3 (Q-\beta_i)^{-1} H
	\begin{pmatrix}
	\beta_1, \beta_2, \beta_3 \\
	\sigma_1,  \sigma_2,   \sigma_3
	\end{pmatrix}$ as desired. 
\end{proof}

\subsection{LCFT description of quantum disks and wedges}\label{sec:pre-lcft-disk}
The goal of this section is to give LCFT descriptions of quantum disks and quantum wedges which have an additional quantum typical boundary point (Propositions~\ref{prop-m2dot} and~\ref{prop:3ptwedge}).

We start with the following definition of the three-pointed quantum disks. 
\begin{definition}\label{three-pointed-disk}
	Fix $W>0$. First sample a 
    quantum disk from the weighted law
    $L\mathcal{M}_2^{\textup{disk}}(W)$, where $L$ is the quantum length of the left boundary arc of the quantum disk. Then sample a point on the left boundary arc from the probability measure proportional to the quantum length measure. We denote the law of this three-pointed quantum surface by $\mathcal{M}_{2, \bullet}^{\textup{disk}}(W)$.
\end{definition}

In~\cite{AHS21}, it is shown that three-pointed quantum disks can be described in terms of three-pointed Liouville fields. On the other hand, our quantum triangles are defined in terms of three-pointed Liouville fields as well. This leads to the following.
\begin{proposition}[Lemma 6.12 of~\cite{ASY22}]\label{prop-m2dot}
  Let $W>0$.  There exists some constant  {$C=C(W)\in(0,\infty)$} such that $\QT(W,W,2) = C\Md_{2,\bullet}(W)$.
\end{proposition}

On the other hand, the quantum wedges can similarly be described via the LCFT uniform embedding. We start from the three-pointed quantum wedge, which is the infinite measure obtained by adding a third point according to the quantum length measure.

\begin{definition}\label{def:3ptwedge-0}
    Let $W>0$, and consider a sample $((S,a_1,a_2),u)$ from  $\mathcal{M}_{2}^{\mathrm{wedge}}(W)\times \mathrm{Leb}_{\bbR_+}$, where $(a_1,a_2)$ are the two marked points of $S$, and  some neighborhood of $a_1$ has finite quantum area. We mark the point $a_3$ on the left boundary of $S$ such that the segment of the left boundary of $S$ between $a_1$ and $a_3$ has quantum  length $u$. We write  $\mathcal{M}_{2,\bullet}^{\mathrm{wedge}}(W)$ as the law of the marked quantum surface $(S,a_1,a_2,a_3)$. The point $a_1$ (resp.\ $a_2$) is called the finite (resp.\ infinite) endpoint of $S$, and the point $a_3$ is called the third marked point of $S$.
\end{definition}

The following is immediate from Definition~\ref{def:3ptwedge-0}.
\begin{lemma}\label{def:3ptwedge}
	Fix $W>\frac{\gamma^2}{2}$. Let $M^0$ be some probability measure on $H^{-1}(\mathcal{S}) $ such that for $\phi\sim M^0$, $(\mathcal{S}, \phi, +\infty, -\infty)/\sim_\gamma$ is a weight $W$ quantum wedge. Let $Q^0$ be the  {unique Borel measure} on $\{(\phi, u): \phi\in H^{-1}(\mathcal{S}),\ u \in\bbR\}$ such that for non-negative functions $f$ and $g$  {measurable with respect to the Borel $\sigma$-algebra on $H^{-1}(\mathcal{S}) $,}
	$$\int f(\phi)g(u)Q^0(d\phi,du) = \int f(\phi)\big(\int_\bbR g(u)\nu_\phi(du)\big)M^0(d\phi).$$
	Then   $\mathcal{M}_{2,\bullet}^{\mathrm{wedge}}(W)$ is the law of the marked quantum surface $(\mathcal{S}, \phi, +\infty, -\infty, u)/{\sim_\gamma}$.
\end{lemma}
 {We will use the following Liouville field description for  $\mathcal{M}_{2,\bullet}^{\mathrm{wedge}}(W)$, which is the analog of~\cite[Proposition 2.18]{AHS21} in the quantum wedge setting.}

\begin{proposition}\label{prop:3ptwedge}
	For $W>\frac{\gamma^2}{2}$,  $\beta = \gamma+\frac{2-W}{\gamma}$, let $\phi$ be sampled from  $\frac{1}{Q-\beta}\textup{LF}_\mathcal{S}^{(2Q-\beta, -\infty), (\beta, +\infty), (\gamma,0)}$.  Then  $(\mathcal{S}, \phi, +\infty, -\infty, 0)/{\sim_\gamma}$ is a sample from  $\mathcal{M}_{2,\bullet}^{\mathrm{wedge}}(W)$.
\end{proposition}

The following statement  {states that when a quantum wedge is embedded in $(\mathcal S, +\infty, -\infty)$ with embedding chosen uniformly at random, the resulting field is a Liouville field.}  {A version of this result for the upper half plane $\bbH$ is stated in \cite[Proposition 5.1]{ang-zipper}, and we further transfer from $\bbH$ to $\mathcal{S}$ by conformal covariance of Liouville fields~\cite[Lemma 2.9]{ASY22}.} 
\begin{proposition}\label{thm-lcft-wedge}
	Fix $W>\frac{\gamma^2}{2}$, $\beta = \gamma+\frac{2}{\gamma}-\frac{W}{\gamma}$ and let $(\mathcal{S}, \phi, +\infty, -\infty)$ be a sample from $\mathcal{M}_2^{\mathrm{wedge}}(W)$. Let $T\in\bbR$ be sampled independently from the Lebesgue measure $dt$ and set $\tilde{\phi}(z) = \phi(z-T)$. Then the law of $\tilde{\phi}$ is given by $\frac{1}{Q-\beta}\mathrm{LF}_\mathcal{S}^{(2Q-\beta, -\infty), (\beta, +\infty)}$. 
\end{proposition}	

\begin{proof}[Proof of Proposition \ref{prop:3ptwedge}]
    	In the proof below, for $\phi\in H^{-1}(\mathcal{S})$, we shall write $(\phi,u)$ for the marked quantum surface $(\mathcal{S}, \phi, +\infty, -\infty, u)/{\sim_\gamma}    $. Let $M^0$ and $Q^0$ be as in Lemma~\ref{def:3ptwedge}. Suppose $\tilde{F}$ is a  {continuous} non-negative functional over $\{(\phi,u): \phi\in H^{-1}(\mathcal{S}),\ u \in\bbR\}$. Then by definition of the equivalence relation $\sim_\gamma$, we can find a functional $F$ on $H^{-1}(\mathcal{S})$ such that $\tilde{F}(\phi, u) = F(\phi(\cdot + u))$ as we can shift to make our marked point located at 0. Then we for $\tilde{\phi} = \phi(\cdot+u)$
	\begin{equation}
	\begin{split}
	\int \tilde{F}(\phi, u)Q^0(d\phi, du)&= \lim_{\e\to 0}\int\int_\bbR F(\phi(\cdot+u))\e^{\frac{\gamma^2}{4}}e^{\frac{\gamma}{2}\phi_\e(u)}duM^0(d\phi)\\
	&= \lim_{\e\to 0}\int\int_\bbR F(\tilde{\phi})\e^{\frac{\gamma^2}{4}} e^{\frac{\gamma}{2}\tilde{\phi}_\e(0)}duM^0(d\phi)\\
	&= \lim_{\e\to 0}\int F(\phi) \e^{\frac{\gamma^2}{4}} e^{\frac{\gamma}{2}{\phi}_\e(0)} \frac{1}{Q-\beta}\text{LF}_\mathcal{S}^{(2Q-\beta, -\infty), (\beta, +\infty)}(d\phi)\\
	&= \frac{1}{Q-\beta}\lim_{\e\to 0}\int_\bbR\int F(h -(Q- \beta)\text{Re}\cdot + c ) \e^{\frac{\gamma^2}{4}} e^{\frac{\gamma}{2}{h}_\e(0)+\frac{\gamma}{2}c} \mathcal{P}_{\mathcal{S}}(dh)dc.
	\end{split}
	\end{equation}
	where in the third line we have used Proposition~\ref{thm-lcft-wedge} that the law of $\tilde{\phi}$ under $du M^0(d\phi)$ is the same as $\frac{1}{Q-\beta}\text{LF}_\mathcal{S}^{(2Q-\beta, -\infty), (\beta, +\infty)}(d\phi)$. Now by  Girsanov's theorem, for fixed $c\in\bbR$, 
	\begin{equation}
	\int F(h -(Q- \beta)\text{Re}\cdot + c ) \e^{\frac{\gamma^2}{4}} e^{\frac{\gamma}{2}{h}_\e(0)+\frac{\gamma}{2}c} \mathcal{P}_{\mathcal{S}}(dh) = \int F\big(h +\frac{\gamma}{2}G_{\mathcal{S}, \e}(\cdot, 0)- (Q- \beta)\text{Re}\cdot + c\big)\e^{\frac{\gamma^2}{4}}\mathcal{P}_{\mathcal{S}}(dh)\bbE[ e^{\frac{\gamma}{2}{h}_\e(0)}]
	\end{equation}
	where $G_{\mathcal{S}, \e}(z, u) = \bbE [h(z)h_\e(u)]\to G_{\mathcal{S}}(z, u)$ as $\e \to0$. Since $\var(h_\e(0)) = -2\log\e + o(1)$, it follows that 
	\begin{equation}
	\begin{split}
	&\lim_{\e\to 0}\int_\bbR\int F(h -(Q- \beta)\text{Re}\cdot + c ) \e^{\frac{\gamma^2}{4}} e^{\frac{\gamma}{2}{h}_\e(0)+\frac{\gamma}{2}c} \mathcal{P}_{\mathcal{S}}(dh)dc\\
	& = \int_\bbR\int F\big(h+\frac{\gamma}{2}G_{\mathcal{S}}(\cdot, 0)-(Q- \beta)\text{Re}\cdot + c\big)e^{\frac{\gamma}{2}c}\mathcal{P}_{\mathcal{S}}(dh)dc.
	\end{split}
	\end{equation}
	Note that by definition, a sample from LF$_\mathcal{S}^{(2Q-\beta, -\infty), (\beta, +\infty), (\gamma,0)}$ is obtained precisely by first taking $(h,\textbf{c})$ from $\mathcal{P}_{\mathcal{S}}\times e^{\frac{\gamma}{2}c}dc$ and then set $\phi = h+\frac{\gamma}{2}G_{\mathcal{S}}(\cdot, 0)-(Q- \beta)\text{Re}\cdot + c$.  {Therefore we conclude that} 
	$$
	\int \tilde{F}(\phi, u)Q^0(d\phi, du) = \frac{1}{Q-\beta}\int   \tilde{F}(\phi, 0)\text{LF}_\mathcal{S}^{(2Q-\beta, -\infty), (\beta, +\infty), (\gamma,0)}(d\phi)
	$$
	which justifies the proposition.
\end{proof}

\subsection{Imaginary geometry flow lines}\label{sec:pre-ig}

In this section we briefly go over the GFF/SLE coupling as constructed in imaginary geometry \cite{MS16a,MS17}.  {See~\cite{BN16} for a detailed introduction to the theory of SLE.} Let $\kappa\in(0,4)$.  Heuristically, given a GFF $h$, a curve $\eta(t)$ is a flow line of angle $\theta$ if
\[ 
\eta'(t) = e^{i(\frac{h(\eta(t))}{\chi}+\theta)}\ \text{for}\ t>0, \ \text{where}\ \chi = \frac{2}{\sqrt{\kappa}}-\frac{\sqrt{\kappa}}{2}.
\]
Similarly to quantum surfaces~\eqref{eqn-qs-relation}, define an \textit{imaginary surface} to be an equivalence class of pairs $(D, h)$ where $(D,h) \sim (\tilde D, \tilde h)$ if there is a conformal map $\psi: D \to \tilde D$ and choice of branch cut for $\arg$ such that 
\eqb\label{eq-ig-surface}
\tilde h = h \circ \psi^{-1} - \chi \arg ((\psi^{-1})').
\eqe

We now introduce the SLE$_\kappa(\underline{\rho})$ curves; as we discuss later, these will arise as flow lines of the GFF. 
Fix  $x^{k,L}<...<x^{1,L}\le 0^-\le 0^+\le x^{1,R}<...<x^{\ell, R}$, which are called \textit{force points}, and set $\underline{x} = (\underline{x}_L, \underline{x}_R) = (x^{1,L}, ..., x^{k,L};x^{1,R}, ..., x^{\ell, R})$. For each   force point $x^{i,q}$, $q\in\{L,R\}$ we assign a \textit{weight} $\rho^{i,q}\in\mathbb{R}$. Let $\underline{\rho}$ be the vector of weights, and  we shall work on the upper half plane. The SLE$_\kappa(\underline{\rho})$ process is a measure on continuously growing compact hulls $K_t$ with the  Loewner driving function $(W_t)_{t\ge 0}$ characterized by 
\begin{equation}\label{eqn-def-sle-rho}
\begin{split}
&W_t = \sqrt{\kappa}B_t+\sum_{q\in\{L,R\}}\sum_i \int_0^t \frac{\rho^{i,q}}{W_s-V_s^{i,q}}ds; \\
& V_t^{i,q} = x^{i,q}+\int_0^t \frac{2}{V_s^{i,q}-W_s}ds, \ q\in\{L,R\},
\end{split}
\end{equation}
while $g_t$ is the unique conformal transformation from $\mathbb{H}\backslash K_t$ to $\mathbb{H}$ such that $\lim_{|z|\to\infty}|g_t(z)-z|=0$ with
\begin{equation}\label{eqn-def-sle}
g_t(z) = z+\int_0^t \frac{2}{g_s(z)-W_s}ds, \ z\in\mathbb{H}.
\end{equation} 
It has been shown in \cite{MS16a} that  SLE$_\kappa(\underline{\rho})$ processes a.s.\ exist, are unique and generate a continuous curve until the \textit{continuation threshold}, the first time $t$ such that $W_t = V_t^{j,q}$ with $\sum_{i=1}^j\rho^{i,q}\le -2$ for some $j$ and $q\in\{L,R\}$.


 {Theorems 1.1 and 1.2 of \cite{MS16a} introduce a coupling of a GFF $h$ in $\bbH$ and an SLE$_\kappa(\underline \rho)$ curve $\eta$ where $\eta$ is measurable with respect to $h$; the curve $\eta$ is called the flow line of $h$. A curve $\eta$ is called a flow line of angle $\theta$ if it is the flow line of $h + \theta\chi$, and we can simultaneously consider flow lines starting from different boundary points and having different angles. Flow lines of a GFF in general simply-connected domains with boundary are defined by uniformizing in $\bbH$ via~\eqref{eq-ig-surface}.}

By \cite[Theorem 1.5]{MS16a}, the interactions of  flow lines (i.e., crossing, merging, etc.) are completely determined by their angles. One important consequence is that, as argued in \cite[Section 6]{MS16a}, if $\eta_1$ and $\eta_2$ are flow lines of $h$ with angles $\theta_1$ and $\theta_2$, then given $\eta_1$, the conditional law of $\eta_2$ is the same as the law of the flow line (with angle $\theta_2$) of the GFF in $\mathbb{H}\backslash\eta_1$ with the \textit{flow line boundary conditions}.  This will be applied later when determining the conditional law of SLE curves.

One can also make sense of flow lines of a whole plane  GFF $\wh h$ started from interior points as in \cite[Theorem 1.1]{MS17}; these flow lines are whole plane $\SLE_\kappa(2-\kappa)$ processes. Moreover, flow lines of $\wh h$ of the same angle $\theta_0$ started from different points a.s.\ merge when they intersect  \cite[Theorem 1.9]{MS17}. This gives an ordering of $\mathbb{Q}^2$, where $z\lesssim w$ whenever the $\theta_0$-angle flow line from $z$ merges into the $\theta_0$-angle flow line from $w$ on the left side.  One can construct a
unique Peano curve $\eta'$ which visits points of $\mathbb Q^2$ with respect to this ordering~\cite[Theorem 1.16]{MS17}. We will work with the case $\theta_0=\frac{\pi}{2}$, for which  $\eta'$ is called the north-going space-filling counterflowline of $\wh h$. It  has the law as the space-filling $\SLE_{16/\kappa}$ in $\mathbb C$ from $\infty$ to $\infty$.


\section{Conformal welding of quantum triangle with  quantum wedge}\label{sec:conf} As pioneered by~\cite{She16a,DMS14}, for certain pairs of quantum surfaces, there exists a way to \emph{conformally weld} them according to their quantum boundary length $\nu_\phi$ measures provided that these boundary lengths agree; see e.g.~\cite[Section 4.1]{AHS21} and~\cite[Section 4.1]{ASY22} for details. Let $\mathcal{M}^1, \mathcal{M}^2$ be  {$\sigma$-finite} measures on the space of   quantum surfaces with boundary marked points. For $i=1,2$, fix a boundary arc $e_i$ of finite quantum length on a sample from $\mathcal{M}^i$, and define the measure $\mathcal{M}^i(\ell_i)$ via the disintegration 
\eqb\label{eq:weld-def}\mathcal{M}^i = \int_0^\infty \mathcal{M}^i(\ell_i)d\ell_i\eqe
as in Section~\ref{sec:lcft-qt}. For $\ell>0$, given a pair of surfaces sampled from the product measure $\mathcal{M}^1(\ell)\times \mathcal{M}^2(\ell)$, we can conformally weld them together according to  quantum length. This yields a single quantum surface decorated  by a curve, namely, the welding interface. We write $\text{Weld}( \mathcal{M}^1(\ell), \mathcal{M}^2(\ell))$ for the law of the resulting curve-decorated surface, and let $$\text{Weld}(\mathcal{M}^1, \mathcal{M}^2):=\int_{0}^\infty\, \text{Weld}(  \mathcal{M}^1(\ell), \mathcal{M}^2(\ell))\,d\ell$$ be the  welding of $\mathcal{M}^1,  \mathcal{M}^2$ along the boundary arcs $e_1$ and $e_2$. If $e_1$ and $e_2$ both have infinite quantum length, then we let $\text{Weld}(\mathcal{M}^1, \mathcal{M}^2)$ be the welding of the two surfaces along the boundary arcs $e_1$ and $e_2$. If $e_1$ has finite quantum length and $e_2$ has infinite quantum length, then we add a marked point on $e_2$ dividing $e_2$ into two arcs such that the one subarc $e_2'$ with finite length has the same quantum length as $e_1$, and define $\text{Weld}(\mathcal{M}^1, \mathcal{M}^2)$ to be the welding of the two surfaces along the boundary arcs $e_1$ and $e_2'$.  {Following~\cite{DMS14,AHS21,ASY22}, the $\mathrm{Weld}$ operator is well-defined for the quantum wedges, quantum disks or quantum triangles.}

The aim of this section is to prove Theorem~\ref{thm:weld-0} below, see Figure~\ref{fig:pthetaweld}. {Recall from Definition~\ref{def:3ptwedge-0} that a sample $(D, h, a_1, a_2, a_3)$ from  $\mathcal M_{2, \bullet}^\textup{wedge}(W)$ is a three-pointed quantum surface,  {where $a_1$ and $a_2$ are the marked points of the quantum wedge with finite and infinite neighborhoods respectively,} and $a_3$ was added according to the quantum boundary length measure. For $W \geq \frac{\gamma^2}2$, let $\mathcal{M}_{2,\bullet}^{\mathrm{wedge}}(W)\otimes \SLE_\kappa(W-2;-W,W-2)$ denote the law of the curve-decorated quantum surface obtained by drawing an independent $\SLE_\kappa(W-2;-W,W-2)$ curve on top of a sample $(D, h, a_1, a_2, a_3)$ from $\mathcal{M}_{2,\bullet}^{\mathrm{wedge}}(W)$; the curve starts from $a_3$ and targets $a_2$,  its first two force points (with weights $W-2$ and $-W$) are located immediately to the left and right hand side of $a_3$, and the third force point (with weight $W-2$) is located at $a_1$. For $W \in (0,\frac{\gamma^2}2)$, the same definition applies except that the curve lies within the connected component $\tilde D$ of $D$ containing $a_3$, and $a_1, a_2$ are replaced by the points $\tilde a_1, \tilde a_2 \in \partial \tilde D$ connecting $\tilde D$ to $a_1$ and $a_2$ respectively. In both cases, since $-W+(W-2)=-2$, the curve reaches its continuation threshold (merges with the boundary) before hitting its target.}

\begin{theorem}\label{thm:weld-0}
    Let $\gamma\in(0,2)$ and $\kappa=\gamma^2$. Let $W\in (0,2-\frac{\gamma^2}{2})\backslash\{\frac{\gamma^2}{2}\}$. Then there exists a constant $C:=C(\gamma,W)$ such that
\eqb\label{eq:thm-weld-0}\mathcal{M}_{2,\bullet}^{\mathrm{wedge}}(W)\otimes \SLE_\kappa(W-2;-W,W-2) = C\, \Wd(\mathcal{M}_{2}^{\mathrm{wedge}}(W),\QT(2-W,W,\gamma^2)).\eqe
 In the welding on the right hand side of~\eqref{eq:thm-weld-0}, we  conformally weld the boundary arc of the quantum triangle between the vertices of weights $2-W$ and $\gamma^2$ to the boundary of the quantum wedge, identifying the weight $2-W$ vertex of the quantum triangle with the finite 
 marked point of the quantum wedge.
\end{theorem}

We expect Theorem~\ref{thm:weld-0} to be true for all $W\in(0,2)$, but we only consider the $W\in (0,2-\frac{\gamma^2}{2})\backslash\{\frac{\gamma^2}{2}\}$ case for technical convenience.

\begin{figure}
    \centering
    \includegraphics[scale=0.85]{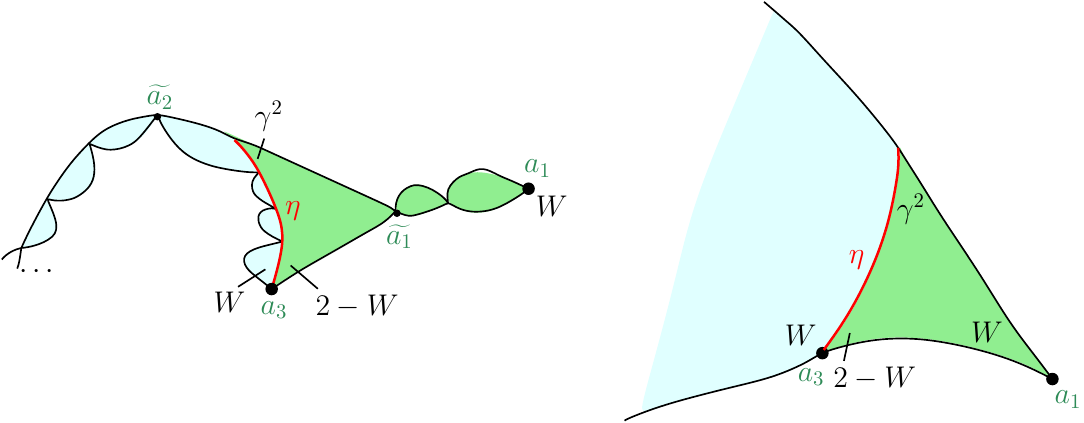}
    \caption{An illustration of Theorem~\ref{thm:weld-0} for $W<\frac{\gamma^2}{2}$ (left) and $W>\frac{\gamma^2}{2}$ (right).  {The point $a_2$ is located at the infinite end of the quantum wedge and is not shown in the figure. Note that $2-W>\frac{\gamma^2}{2}$ since $W\in(0,2-\frac{\gamma^2}{2})$.  {The light blue part is the weight $W$ quantum wedge, while the green part stands for the weight $(\gamma^2,2-W,W)$ quantum triangle.}}}
    \label{fig:pthetaweld}
\end{figure}

In Section \ref{sec:pre-conf}, we recall the conformal welding of quantum wedges in \cite{DMS14} and quantum triangles in \cite{ASY22}. We prove Theorem~\ref{thm:weld-0} in Section~\ref{subsec:pf-thm:weld-0}.

\subsection{Conformal welding of quantum surfaces}\label{sec:pre-conf}

In this section we briefly review the conformal welding of quantum wedges and quantum triangles which will later be applied in our argument.

 For a measure $\mathcal{M}$ on the space of quantum surfaces (possibly with marked points) and a conformally invariant measure $\mathcal{P}$ on curves, we write $\mathcal{M}\otimes\mathcal{P}$ for the law of curve decorated quantum surface described by sampling $(S,\eta)$ from  $\mathcal{M}\times\mathcal{P}$ and then drawing $\eta$ on top of $S$. To be more precise, for a domain with marked points $\cD = (D,z_1,...,z_n)$, suppose $\mathcal{M}_{\cD}$ is a measure such that  for $\phi$ sampled from $\mathcal M_{\cD}$ the law of $(D,\phi,z_1,...,z_n)/{\sim_\gamma}$ is $\mathcal{M}$. Let $\mathcal{P}_{\cD}$ be the measure $\mathcal{P}$ on the domain $\cD$,  and assume that for any conformal map $f$ one has $\mathcal{P}_{f\circ\cD} = f\circ \mathcal{P}_{\cD}$, i.e., $\mathcal{P}$ is invariant under conformal maps. Then $\mathcal{M}\otimes\mathcal{P}$ is defined by  $(D,\phi,\eta,z_1,...,z_n)/{\sim_\gamma}$ for $\eta\sim\mathcal{P}_{\cD}$. This notion is well-defined for the quantum surfaces and SLE-type curves considered in this paper  {thanks to the definition of $\sim_\gamma$ and the conformal invariance of SLE}.

We begin with the conformal weldings of two quantum wedges and two quantum disks.
\begin{theorem}[Theorem 1.2 of \cite{DMS14}]\label{thm:wedge-welding}
    Let $\gamma\in(0,2),\kappa=\gamma^2$ and $W_1,W_2>0$. Then  
$$\mathcal{M}^{\rm wedge}_2(W_1+W_2)\otimes \SLE_{\kappa}(W_1-2;W_2-2) =  \Wd(\mathcal{M}^{\rm wedge}_2(W_1),\mathcal{M}^{\rm wedge}_2(W_2)).  $$
\end{theorem}

\begin{theorem}[Theorem 2.2 of \cite{AHS20}]\label{thm:disk-welding}
    Let $\gamma\in(0,2),\kappa=\gamma^2$ and $W_1,W_2>0$. Then there exists a constant $c:=c_{W_1,W_2}\in(0,\infty)$ such that
$$\Md_2(W_1+W_2)\otimes \SLE_{\kappa}(W_1-2;W_2-2) = c\,\Wd(\Md_2(W_1),\Md_2(W_2)).  $$
\end{theorem}

\begin{figure}[ht]
	\centering
	\begin{tabular}{ccc} 
			\includegraphics[scale=0.7]{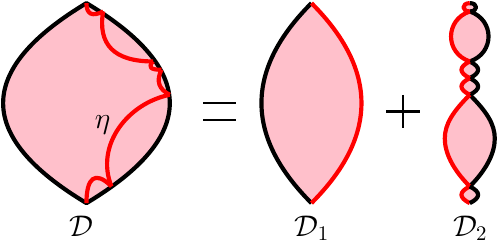}
			& & \ \ \ \
			\includegraphics[scale=0.5]{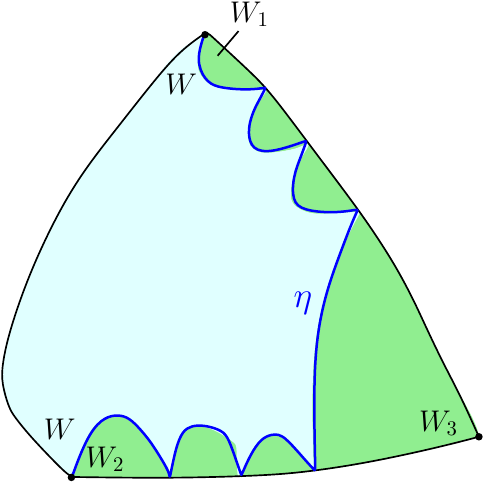}
		\end{tabular}
	\caption{\textbf{Left}: Illustration of Theorem~\ref{thm:disk-welding} with $W_1\ge\frac{\gamma^2}{2}$ and $W_2<\frac{\gamma^2}{2}$. \textbf{Right}: Illustration of Theorem~\ref{thm:disk+QT} with $W,W_3\ge\frac{\gamma^2}{2}$ and $W_1,W_2<\frac{\gamma^2}{2}$.}\label{fig-qt-weld} 
\end{figure}

  {In the settings above, the SLE curves start from one of the marked point on the surface and target at the other marked point, while the force points  are located infinitesimally to the left and right of their starting points, respectively.}  If $W_1+W_2<\frac{\gamma^2}{2}$, then $\Md_2(W_1+W_2)\otimes \SLE_{\kappa}(W_1-2;W_2-2)$ is understood as drawing independent $ \SLE_{\kappa}(W_1-2;W_2-2)$ curves in each bead of the weight $W_1+W_2$ disk, and the $\SLE_{\kappa}(W_1-2;W_2-2)$ is defined by their concatenation.  More explicitly, it is obtained by replacing the measure $\Md_2(\gamma^2-W_1-W_2)$ with $\Md_2(\gamma^2-W_1-W_2)\otimes \SLE_{\kappa}(W_1-2;W_2-2)$ in the Poisson point process construction of $\Md_2(W_1+W_2)$ in Definition~\ref{def-thin-disk}.  

For a quantum triangle of weights $W+W_1,W+W_2,W_3$ with $W_2+W_3 = W_1+2$ embedded as $(D,\phi,a_1,a_2,a_3)$, we start by making sense of the $\SLE_{\kappa}(W-2;W_1-2,W_2-W_1)$ curve $\eta$ from $a_2$ to $a_1$. If the domain $D$ is simply connected (which corresponds to the case where $W+W_1,W+W_2, W_3\ge\frac{\gamma^2}{2}$), $\eta$ is just the ordinary $\SLE_{\kappa}(W-2;W_1-2,W_2-W_1)$  with force points at $a_2^-,a_2^+$ and $a_3$. Otherwise, let $(\tilde{D}, \phi,\tilde{a}_1,\tilde{a}_2, \tilde{a}_3)$ be the thick quantum triangle component as in Definition~\ref{def-qt-thin}, and sample an $\SLE_{\kappa}(W-2;W_1-2,W_2-W_1)$ curve $\tilde{\eta}$ in $\tilde{D}$ from $\tilde{a}_2$ to $\tilde{a}_1$ as above. Then our curve $\eta$ is the concatenation of $\tilde{\eta}$ with independent $\SLE_{\kappa}(W-2;W_1-2)$ curves in each bead of the weight $W+W_1$   quantum disk (if $W+W_1<\frac{\gamma^2}{2}$) and $\SLE_{\kappa}(W-2;W_2-2)$ curves in each bead of the weight $W+W_2$  quantum disk (if $W+W_2<\frac{\gamma^2}{2}$). In other words, if $W+W_1<\frac{\gamma^2}{2}$ and $W+W_2<\frac{\gamma^2}{2}$, then the notion $\QT(W_1,W_2,W_3)\otimes\SLE_{\kappa}(W-2;W_1-2,W_2-W_1)$ is defined through $\QT(\gamma^2-W-W_1,\gamma^2-W-W_2,W_3)\otimes \SLE_{\kappa}(W-2;W_1-2,W_2-W_1)$, $\Md_2(W+W_1)\otimes\SLE_{\kappa}(W-2;W_1-2)$ and $\Md_2(W+W_2)\otimes\SLE_{\kappa}(W-2;W_2-2)$ as in Definition~\ref{def-qt-thin}, while other cases follows similarly.

With this notation, we state the welding of quantum disks with quantum triangles below.
\begin{theorem}[Theorem 1.1 of \cite{ASY22}]\label{thm:disk+QT}
Let $\gamma\in(0,2)$ and $\kappa=\gamma^2$. Fix $W,W_1,W_2,W_3>0$ such that $W_2+W_3=W_1+2$. There exists some constant $c:=c_{W,W_1,W_2,W_3}\in (0,
\infty)$ such that
\begin{equation}\label{eqn:disk+QT}
    \QT(W+W_1,W+W_2,W_3)\otimes \SLE_{\kappa}(W-2;W_2-2,W_1-W_2) = c\Wd (\Md_2(W),\QT(W_1,W_2,W_3)).
\end{equation}
\end{theorem}

\subsection{Proof of Theorem~\ref{thm:weld-0}}\label{subsec:pf-thm:weld-0}

In this section we prove Theorem~\ref{thm:weld-0}. We will first prove a variant of Theorem~\ref{thm:weld-0} for quantum disks (Proposition~\ref{prop:disk+QT-0}), then transfer to the setting of quantum wedges to prove Theorem~\ref{thm:weld-0}. We will adapt the proof in~\cite[Section 3.2]{ASYZ24}.

Our starting point is Proposition~\ref{prop:disk+QT-0} below,  {which is an extension of Theorem~\ref{thm:disk+QT} for $W_1=0$. A variant where $(W_1,W_2,W_3)=(0,\frac{\gamma^2}{2},2-\frac{\gamma^2}{2})$ is shown in~\cite[Proposition 3.6]{ASYZ24}. See Figure~\ref{fig:weld-qt-0}. The proof is almost identical and we include it here for completeness.}

 {We start with the following definition for weight zero quantum surfaces. Roughly speaking, the weight $W$ quantum disks shrinks to a single line segment as $W\downarrow 0$.}

\begin{definition}[Weight zero quantum disks and quantum triangles]
We define the weight zero quantum disk to be a line segment  modulo homeomorphisms of $\bbR^2$ parametrized  by quantum length where the total length is $\mathbf {t}\sim \mathds{1}_{t>0}dt$, and write $\Md_{2}(0)$ for its law.     For $W_1,W_2,W_3\ge0$ where one or more of $W_1,W_2,W_3$ is zero, we define the measure $\QT(W_1,W_2,W_3)$ using $\Md_2(0)$ in the same way as Definition~\ref{def-qt-thin}.
\end{definition}

\begin{figure}
    \centering
    \includegraphics[scale=0.6]{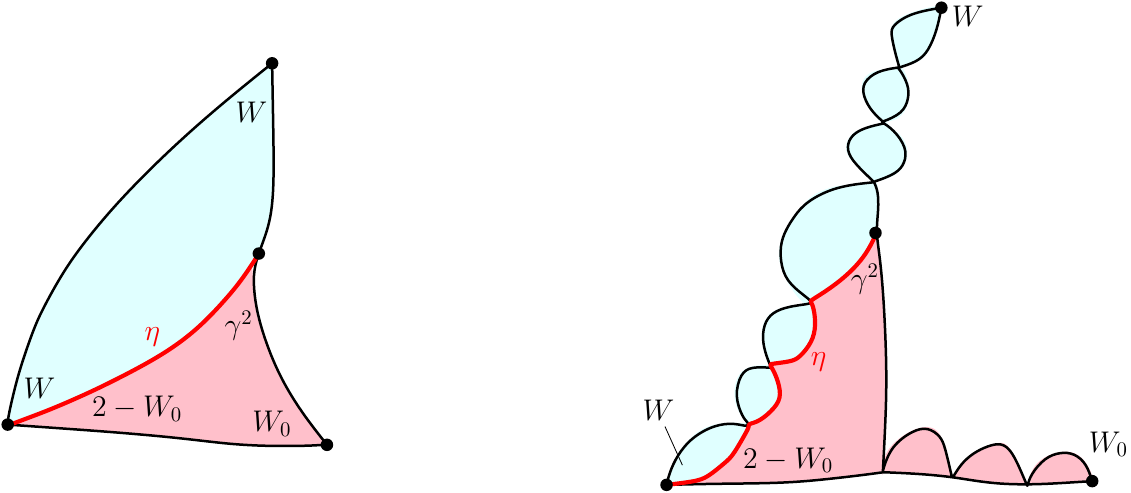}
    \caption{An illustration of Proposition~\ref{prop:disk+QT-0} with 
    {$W_0, W > \frac{\gamma^2}2$} 
    (left) and $W_0,W<\frac{\gamma^2}{2}$ (right).}
    \label{fig:weld-qt-0}
\end{figure}


\begin{lemma}\label{lem:prop:disk+QT-0}
 Let  $(W_1,W_2,W_3) = (0,2-W_0,W_0)$.  There exists some measure  {$\mathsf{m}_{W, W_0}$} on curves such that 
    \begin{equation}\label{eqn:disk+QT-A0}
   \QT(W+W_1,W+W_2,W_3)\otimes \mathsf{m}_{W,W_0} = c\Wd (\Md_2(W),\QT(W_1,W_2,W_3)).
\end{equation}
\end{lemma}

\begin{proof}
    We first prove the statement with $W_3$ replaced by $4-\gamma^2-W_0$.  As in the proof of~\cite[Equation (3.5)]{ASYZ24}, by marking the point on the weight $W$ quantum disk on the interface with distance $r$ to the root, and applying Proposition~\ref{prop-m2dot}, we have
    \begin{equation}\label{eq:disk+QT-0-1}
    \begin{split}
        &\Wd (\Md_{2}(W),\QT(0,2-W_0,4-\gamma^2-W_0) )\\&= c\int_0^\infty\int_0^\ell \Wd\big(\QT(W,2,W;\ell-r,r),\QT(\gamma^2,2-W_0,4-\gamma^2-W_0;r)\big)\, dr\,d\ell\\
        &= c\int_0^\infty \Wd\big(\QT(W,2,W;r), \QT(\gamma^2,2-W_0,4-\gamma^2-W_0;r)\big)\,dr.
        \end{split}
    \end{equation}
    By~\cite[Corollary 4.11]{SY23},  {the output surface from the conformal welding above can be described in terms of a Liouville field with 4 insertions, with one of the $\beta$-insertions having weight $2+\gamma^2$ and the $\beta$-parameter equal to  $\gamma+\frac{2-(2+\gamma^2)}{\gamma} = 0$ and thus can be neglected.}  It then follows that there exists a probability measure $\widetilde{\mathsf{m}}_{W, W_0}$ on curves such that~\eqref{eq:disk+QT-0-1} is equal to a constant times $\QT(W,W+2-W_0,4-\gamma^2-W_0)\otimes\wt{\mathsf{m}}_{W, W_0}$. The claim~\eqref{eq:disk+QT-0-1}   then follows from the same argument in~\cite[Proposition 6.3]{ASY22},  {where we can change  back from $4-\gamma^2-W_0$ at the beginning of the proof to $W_3$}  {via reweighting the law of the underlying Liouville field near the vertex of weight $4-\gamma^2-W_0$; see also the proof of Theorem~\ref{thm:weld-0} below for another example of this argument.}
\end{proof}

 
\begin{lemma}\label{lem:prop:disk+QT-1}
    The measure $\mathsf{m}_{W,W_0}$ in Lemma~\ref{lem:prop:disk+QT-0} is finite.
\end{lemma}
\begin{proof}  {We first prove the claim for $W = 2$, then use that case to address general $W$.
\\
\textbf{Case $W = 2$:}} We first address the $W_0 < \frac{\gamma^2}2$ regime. By Definition~\ref{def-qt-thin}, ~\eqref{eqn:disk+QT-A0} holds for $W_3=\gamma^2-W_0$ instead of $W_3=W_0$, and the measure $\mathsf{m}$ is identical in the two settings. Thus, we can assume $W_3=\gamma^2-W_0$ instead. Consider the event $E$ where after welding, the boundary length of the edge between the weight $W$ and weight $W+2-W_0$ vertices in the weight $(W,W+2-W_0,\gamma^2-W_0)$ quantum triangle is between 1 and 2. By Proposition~\ref{prop:QT-bdry-length}  {the $\QT(W+W_1,W+W_2,W_3)$-law of this length is a power law, hence}   $\big(\QT(W+W_1,W+W_2,W_3)\otimes \mathsf{m}_{W,W_0}\big)[E]$ is a finite number times the size of $\mathsf{m}_{W, W_0}$. On the other hand, arguing exactly as in~\eqref{eq:disk+QT-0-1},
    \begin{equation}\label{eq:disk+QT-0-1-1}
    \begin{split}
        \Wd (\Md_{2}(W),\QT(0,2-W_0,\gamma^2-W_0) )= c\int_0^\infty \Wd\big(\QT(W,2,W;r), \QT(\gamma^2,2-W_0,\gamma^2-W_0;r)\big)\,dr.
        \end{split}
    \end{equation}
     {We now use inputs specific to $W=2$.} 
    Consider the disintegration of $\QT(2,2,2)$ over the three boundary lengths. Then on the right hand of~\eqref{eq:disk+QT-0-1-1}, the event $E$ has measure
    \begin{equation}\label{eq:disk+QT-0-2}
       c\int_1^2\int_0^\infty\int_0^\infty |\QT(2,2,2;\ell_1,\ell_2,r)|\cdot |\QT(\gamma^2,2-W_0,\gamma^2-W_0;r)|dr\,d\ell_2\,d\ell_1.
    \end{equation}
    By~\cite[Proposition 7.8]{AHS20},  {we have $|\Md_{2,\bullet}(2;\ell,r)| = (\ell+r)^{-\frac{4}{\gamma^2}-1}$.} Since $\QT(2,2,2) = C\Md_{2,\bullet}(2)$  {and by definition  the third marked point on quantum disks from $\Md_{2,\bullet}(2)$ is sampled according to the quantum length measure}, $|\QT(2,2,2;\ell_1,\ell_2,r)| = (\ell_1+\ell_2+r)^{-1-\frac{4}{\gamma^2}}$. By Proposition~\ref{prop:QT-bdry-length}, $|\QT(\gamma^2,2-W_0,\gamma^2-W_0;r)| = c_1 r^{\frac{2W_0}{\gamma^2}-1}$ for some $c_1\in(0,\infty)$. Therefore~\eqref{eq:disk+QT-0-2} is a constant times
    $$\int_1^2\int_0^\infty\int_0^\infty (\ell_1+\ell_2+r)^{-\frac{4}{\gamma^2}-1}r^{\frac{2W_0}{\gamma^2}-1}dr\,d\ell_2\,d\ell_1<\infty$$
    since $W_0<\frac{\gamma^2}{2}$. Therefore the measure $\mathsf{m}_{W, W_0}$ is finite.

    For the $W_0>\frac{\gamma^2}{2}$ regime, the argument is identical except that we do not replace $W_0$ with $\gamma^2-W_0$, and we use $|\QT(\gamma^2,2-W_0, W_0;r)| = c_2r^0=c_2$ for some $c_2\in(0,\infty)$.
    \medskip 

    \noindent 
    \textbf{Case $W <  2$:}
    Consider the conformal welding 
    $$\iint_{\bbR_+^2}\Wd\bigg(\Md_{ 2}(2-W;\ell_1),\Md_{2}(W;\ell_1,\ell_2),\QT(0,2-W_0,W_0;\ell_2)\bigg)\,d\ell_1d\ell_2.$$ 
    By Theorem~\ref{thm:disk-welding}, when embedded in a domain $D$, the  two interfaces $(\eta_{\mathrm{L}},\eta_\mathrm{R})$ can be realized by (i) sampling   $\eta_\mathrm{R}$ from $\mathsf{m}_{2,W_0}$ as the interface on the right and (ii) sampling $\eta_{\mathrm{L}}$ as an $\SLE_\kappa(-W;W-2)$ curve in the connected component $D\backslash\eta_{\mathrm{R}}$ to the left hand side of $\eta_{\mathrm{R}}$. Then the joint law of $(\eta_{\mathrm{L}},\eta_\mathrm{R})$  is clearly finite, and $\mathsf{m}_{W, W_0}$ describes the conditional law of $\eta_\mathrm{R}$ given $\eta_\mathrm{L}$, which is clearly finite. 
    
    \medskip 
    \noindent \textbf{Case $W > 2$:}
    Consider the conformal welding  $$\iint_{\bbR_+^2}\Wd\bigg(\Md_{ 2}(W-2;\ell_1),\Md_{2}(2;\ell_1,\ell_2),\QT(0,2-W_0,W_0;\ell_2)\bigg)\,d\ell_1d\ell_2.$$ By Theorem~\ref{thm:disk+QT},   the  two interfaces $(\eta_{\mathrm{L}},\eta_\mathrm{R})$ can be realized by (i) sampling   $\eta_\mathrm{L}$ from $\SLE_\kappa(W-4;2-W_0,W_0-2)$ as the interface on the left and (ii) sampling   $\eta_\mathrm{R}$ from $\mathsf{m}_{2,W_0}$ in the connected component $D\backslash\eta_{\mathrm{L}}$ to the right hand side of $\eta_{\mathrm{L}}$. Then $\mathsf{m}_{W, W_0}$ is the marginal law of  $\eta_\mathrm{R}$, which is also finite. This completes the proof. 
\end{proof}
\begin{figure}
    \centering
    \includegraphics[scale=0.6]{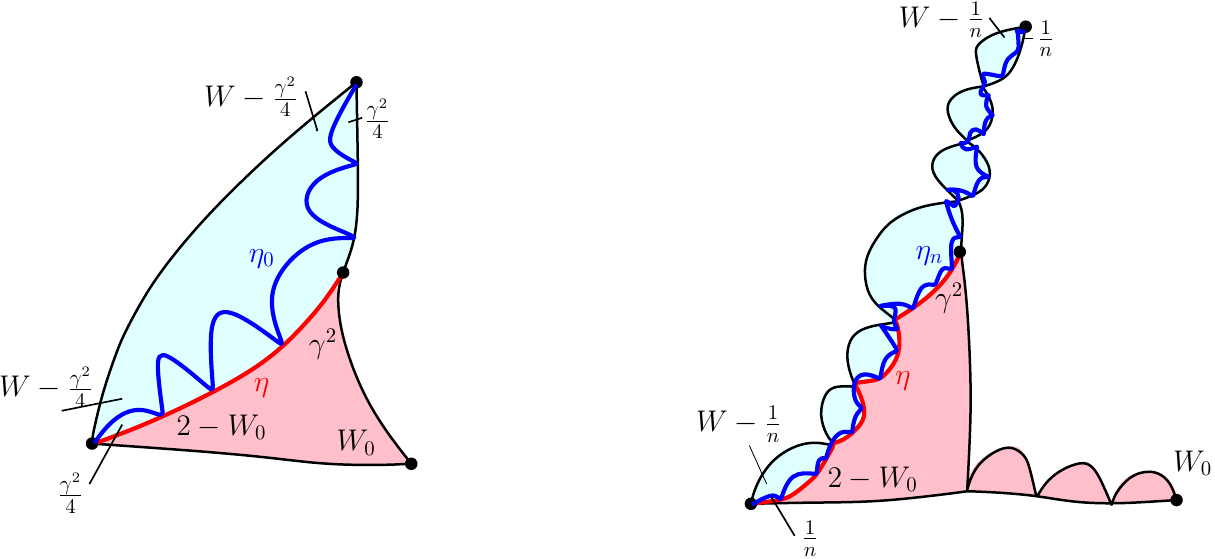}
    \caption{ {An illustration of the proof of Proposition~\ref{prop:disk+QT-0} for $W\geq\frac{\gamma^2}2$ (left) or $W<\frac{\gamma^2}{2}$ (right).}}
    \label{fig:(1)}
\end{figure}

\begin{proposition}\label{prop:disk+QT-0}
 Let $W>0$  {be any positive number} and $W_0\in(0,2-\frac{\gamma^2}{2})\backslash\{\frac{\gamma^2}{2}\}$.   Theorem~\ref{thm:disk+QT} holds for $(W_1,W_2,W_3) = (0,2-W_0,W_0)$.
\end{proposition}
 {
\begin{proof}
Lemma~\ref{lem:prop:disk+QT-0} yields Proposition~\ref{prop:disk+QT-0} but without identifying the law of the curve, and  Lemma~\ref{lem:prop:disk+QT-1}  shows finiteness of the curve measure.
The curve $\eta$ is then identified as $\SLE_{\kappa}(W-2;W_2-2,W_1-W_2)$ by an argument identical to that of \cite[Proposition 3.6]{ASYZ24}; see the last two paragraphs of that proof  {and also Figure~\ref{fig:(1)} for the setup}. 
For $W < \frac{\gamma^2}2$, the law of the curve is identified by taking limits in Theorem~\ref{thm:disk+QT},  {where by Theorem~\ref{thm:disk-welding} we further cut the weight $W$ quantum disk into a weight $W-\frac1n$ and a weight $\frac1n$ quantum disk by an interface $\eta_n$. The interface $\eta_n$ has the law $\SLE_{\kappa}(W-2-\frac{1}{n};W_2-2+\frac{1}{n},W_1-W_2)$ and converges to $\eta$ in Hausdorff topology if we embed the welded surface in a compact domain, and we identify the law of $\eta$ by sending $n\to\infty$ using the continuity of the Loewner chains (see e.g.~\cite[Section 6.1]{Kem17SLE}).} For $W \geq \frac{\gamma^2}2$ the result is obtained from the previous case by  Theorems~\ref{thm:disk-welding} and~\ref{thm:disk+QT} via cutting the weight $W$ quantum disk into quantum disks of smaller weights  {by a curve $\eta_0$, where the marginal law of $\eta_0$ and the conditional law of $\eta$ given $\eta_0$ follows from Theorem~\ref{thm:disk+QT} together with the previous case, which determines the marginal law of the interface $\eta$ following the imaginary geometry theory~\cite[Theorem 1.1]{MS16a}.}
\end{proof}
}

\begin{proof}[Proof of Theorem~\ref{thm:weld-0}]
    We start with the case $W>\frac{\gamma^2}{2}$.   {We start from the setting of Proposition~\ref{prop:disk+QT-0} for quantum disks and use the Liouville field description in  Lemma~\ref{lm:bdry-length-law} and Proposition~\ref{prop:3ptwedge} as well as the argument in~\cite[Proposition 4.5]{AHS21} for changing Liouville field insertions to transfer to the setting for quantum wedges.}
    Consider the setting of Proposition~\ref{prop:disk+QT-0} with $W_0=W$.
    Embed the entire surface  {(i.e., the curve-decorated surface from the left hand side of~\eqref{thm:weld-0})} as $(\mathcal{S},\phi,\eta,-\infty,0,+\infty)$ such that the weight  {$(W,W,2)$ vertices are located at $(-\infty,0,+\infty)$}, and the interface $\eta$ starts from $+\infty$ and targets $0$, with force points at $(+\infty)^\pm$  {(i.e., infinitesimally to the counterclockwise/clockwise direction of $+\infty$)} and $-\infty$. Note that $\eta$ merges with $(-\infty,0)$ at some point $x_\eta<0$ before hitting $0$. By~\eqref{eq:disk+QT-0-1}, $(\mathcal{S},\phi,0,-\infty,+\infty)/{\sim_\gamma}$ has the same law as
    \begin{equation}
        c\int_0^\infty \Wd\big(\QT(W,2,W;r), \QT(\gamma^2,2-W,W;r)\big)\,dr.
    \end{equation}
    Let $D_\eta$ be the connected component of $\mathcal{S}\backslash\eta$ with $-1$ on the boundary, and  $\psi_\eta:D_\eta\to\mathcal{S}$ be a conformal map sending $(+\infty,0,x_\eta)$ to $(+\infty,0,-\infty)$. Let $\varphi_\eta: \mathcal{S}\backslash D_\eta\to\mathcal{S}$ be the conformal map sending $(+\infty,-\infty,x_\eta)$   to $(-\infty,0,+\infty)$. We set \eqb\label{eq:pf-thm-weld-000} X = \phi\circ\psi_\eta^{-1}+Q\log|(\psi_\eta^{-1})'|, \ \ Y = \phi\circ\varphi_\eta^{-1}+Q\log|(\varphi_\eta^{-1})'|.\eqe Then given the interface length $r$, the law of $X$ is a constant times $\LF_{\mathcal{S},r}^{(\beta,-\infty),(\gamma,+\infty),(\beta,0)}$ as in Lemma~\ref{lm:bdry-length-law}, where $\beta = \gamma+\frac{2-W}{\gamma}$. Now we weight the law of $(\phi,\eta)$ by $\e^{\frac{(2Q-\beta)^2-\beta^2}{4}}e^{\frac{(2Q-\beta)-\beta}{2}X_\e(0)}$. Using the identical argument as in the proof of~\cite[Proposition 4.5]{AHS21} with~\cite[Lemma 4.6]{AHS21} as input,  {where the reweighting above changes the size of the  parameter for the insertion at the point 0 of the Liouville field $X$ from $\beta$ to $2Q-\beta$ thanks to  Girsanov's theorem,}   as we send $\e\to 0$,  we see that for $(\phi,\eta)\sim \LF_{\mathcal{S}}^{(\beta,-\infty),(\gamma,+\infty),(2Q-\beta,0)} \times \SLE_\kappa(W-2;-W,W-2)$, the pair $(X,Y)$ defined in~\eqref{eq:pf-thm-weld-000} has the law
    $$c\int_0^\infty \LF_{\mathcal{S},r}^{(\beta,-\infty),(\gamma,+\infty),(2Q-\beta,0)}(dX)\times \LF_{\mathcal{S},r}^{(2\gamma+\frac{2}{\gamma}-\beta,-\infty), (\frac{2}{\gamma},+\infty),(\beta,0)}(dY)\,dr.$$
     {Since $\phi$ and $X$ are related via~\eqref{eq:pf-thm-weld-000}, the reweighting also changes the size of the parameter for the insertion at the point 0 of the Liouville field $\phi$ from $\beta$ to $2Q-\beta$ thanks to  Girsanov's theorem. This will also possibly induce a weighting on the law of the interface $\eta$, but as computed in~\cite[Proposition 4.5]{AHS21},
     we have $\Delta_\beta=\Delta_{2Q-\beta}$ and there is no actual weighting on the law of $\eta$.}
    Further by conformal covariance of Liouville fields~\cite[Lemma 2.9, Lemma 2.11]{ASY22} along with Proposition~\ref{prop:3ptwedge}, for $X \sim \LF_{\mathcal{S}}^{(\beta,-\infty),(\gamma,+\infty),(2Q-\beta,0)}$  the quantum surface $(\mathcal S, X, -\infty, 0, + \infty)$ has law  $c\mathcal{M}_{2,\bullet}^{\mathrm{wedge}}(W)$ for some constant $c_0$.  {Thus, by Definition~\ref{def:3ptwedge-0}, for $X \sim c_0^{-1}\LF_{\mathcal{S}, r}^{(\beta,-\infty),(\gamma,+\infty),(2Q-\beta,0)}$  the quantum surface $(\mathcal S, X, -\infty, 0, + \infty)$ is a sample from $\mathcal M_{2}^\textup{wedge}(W)$ with a third marked point added at quantum boundary length $r$ from the first marked point.} 
    This verifies~\eqref{eq:thm-weld-0} for $W>\frac{\gamma^2}{2}$. 

    For $W<\frac{\gamma^2}{2}$, by~\cite[Proposition 4.2 and Proposition 4.4]{AHS20},  {which gives a decomposition of thin quantum disks/wedges with a third marked point,} for some constants $c,c'$,
    \begin{align}
    \label{eq-marked-decomp-1}
   & \Md_{2,\bullet}(W) = c\Md_2(W)\times\Md_{2,\bullet}(\gamma^2-W)\times \Md_2(W), \\ \label{eq-marked-decomp-2} &\mathcal{M}_{2,\bullet}^{\mathrm{wedge}}(W) = c'\Md_2(W)\times\Md_{2,\bullet}(\gamma^2-W)\times \mathcal{M}_2^{\rm wedge}(W).
    \end{align}
    Consider the setting of Proposition~\ref{prop:disk+QT-0} with $W_0 = W$. 
   Using $\QT(W,2,W)=c\Md_{2,\bullet}(W)$ and~\eqref{eq-marked-decomp-1}, the welded surface can be viewed as concatenation of surfaces $(S_1,S_2,S_3)$ from $c\Md_2(W)\times (\Md_{2,\bullet}(\gamma^2-W)\otimes \SLE_\kappa(W-2;-W,W-2))\times \Md_2(W)$, where $S_3$ is concatenated with $S_2$ at the target point of the interface. Using~\eqref{eq-marked-decomp-2}, the claim is then immediate once we replace $S_3$ with a weight $W$ quantum wedge.
\end{proof}

\section{The constant $p$ and LCFT}\label{sec:p-lcft}
In this section, we express the constant $p$ via the angle $\theta$ and the area of quantum triangles. Throughout the section, fix $\gamma\in(0,2)$, $\kappa=\gamma^2$, $\chi=\frac{2}{\gamma}-\frac{\gamma}{2}$ and $\kappa'=\frac{16}{\kappa}$. In Section \ref{sec:p-contour}, we review the setup of the contour decomposition of the mating of trees and the explicit description of the $p$-$\theta$ relation as in \cite{GHS16}. In Section \ref{sec:p-quantum-area}, we express the constant $p$ in terms of the quantum area of the space-filling SLE$_{\kappa'}$ curve $\eta'$ and further express $p$ in terms of the quantum area of quantum triangles as in Lemma~\ref{lem:p-qt-A} by applying the conformal welding in Theorem \ref{thm:weld-0}. In Section~\ref{subsec:area-lcft}, we calculate the quantum area of these quantum triangles using exact formulas for three-point functions in LCFT as derived in~\cite{arsz-structure-constants}. Finally in Section~\ref{subsec:p-theta-result}, we derive the formula for $p_\gamma(\theta)$.

\subsection{The contour decomposition}\label{sec:p-contour}

We review the contour decomposition and the setup of the $p$-$\theta$ relation as in \cite[Section 3]{GHS16}. Let $(\mathbb{C}, h, 0, \infty)$ be a weight $4-\gamma^2$ quantum cone and $\widehat{h}$ be an independent whole plane GFF. Let $\eta'$ be the north-going space-filling  counterflowline of $\wh h$ introduced in Section~\ref{sec:pre-ig}, parameterized by the $\gamma$-quantum area induced by $h$ and normalized such that $\eta'(0)=0$. 
For $t\in \bbR$, let $L_t$ (resp.\ $R_t$) be change of the $\gamma$-quantum length of the left (resp.\ right) boundary of $\eta'((-\infty,t])$ with respect to time 0. By \cite[Theorem 1.9]{DMS14}, the process $Z = (L, R)$ is a Brownian motion with    \eqb\label{eqn:mot-covariance} \mathrm{Cov}(L_t,R_t) = -\cos(\frac{\pi\gamma^2}{4})\mathbf{a}^2|t|, \ \ \ \ \mathrm{Var}(L_t)=\mathrm{Var}(R_t) = \mathbf{a}^2|t|.\eqe
The variance \eqb\label{eq:mot-var-a} \mathbf{a} = \sqrt{2}\sin(\frac{\pi\gamma^2}{4})^{-1/2}\eqe
is  computed in~\cite{AGS21}.

Now fix the angle $\theta\in (-\frac{\pi}{2},\frac{\pi}{2})$ and consider the angle $\theta$ flow line $\eta_\theta$ of $\widehat{h}$ started from 0. For $t\ge0$, we say that \emph{$\eta'$ crosses $\eta_\theta$} at time $t$ if for any $\e>0$, one can find $s_1,s_2\in (t-\e, t+\e)$ such that $\eta'(s_1)$ and $\eta'(s_2)$ lie on different sides of $\eta_\theta$. Then we define 
$$\mathcal{A}:=\{t\ge 0:\eta' \ \text{crosses}\ \eta_\theta\ \text{at time}\ t\}.$$

The starting point is the following contour decomposition.
\begin{proposition}[Proposition 3.2 of \cite{GHS16}]\label{prop:contour}
For $s\ge0$, let $\tau_s$ (resp.\ $\tilde{\tau}_s$) be the smallest $t\ge s$ (resp.\ largest $t\le s$) for which $t\in\mathcal{A}$. Let $E_s$ be the event that $\eta'(s)$ lies to the left of $\eta_\theta$ and 
\begin{equation}\label{eqn:def-X_s}
    X_s:=(R_s-R_{\tilde{\tau}_s})\mathds{1}_{E_s} - (L_s-L_{\tilde{\tau}_s})\mathds{1}_{E_s^c}, \ \ s\ge 0. 
\end{equation}
There is a deterministic constant $p=p_\gamma(\theta)\in (0,1)$ depending only on $\theta$ and $\gamma$ such that the following holds. The function $|X|$ has the law of $\mathbf{a}$ times a standard reflected Brownian motion and 
$$\mathcal{A}=X^{-1}(0)=\{\tau_s,\tilde{\tau}_s:s\ge 0\}.$$
Moreover, given $\mathcal{A}$, the conditional law of the excursions $\{(Z-Z_{\tilde{\tau}_s})|_{[\tilde{\tau}_s,\tau_s]}:s\ge 0\}$ is that of a collection of independent random Brownian paths. Each path is equal with probability $p$ (resp.\ $1-p$) to a Brownian motion with covariance as in \eqref{eqn:mot-covariance} started from 0 conditioned to stay in the upper (resp.\ right) half plane for $\tau_s-\tilde{\tau}_s$ units of time and to exit the upper (resp.\ right) half plane at time $\tau_s-\tilde{\tau}_s$.
\end{proposition}

From the description above, the constant $p_\gamma(\theta)$ in Proposition~\ref{prop:contour} equals the probability that $X_1>0$, and further equals the probability that $\eta'(1)$ lies on the left hand side of $\eta_\theta$ as in Theorem~\ref{thm:main};  {see Figure~\ref{fig-contour-decomposition}.}

\begin{figure}[ht]
	\centering
	\includegraphics[width=0.66\textwidth]{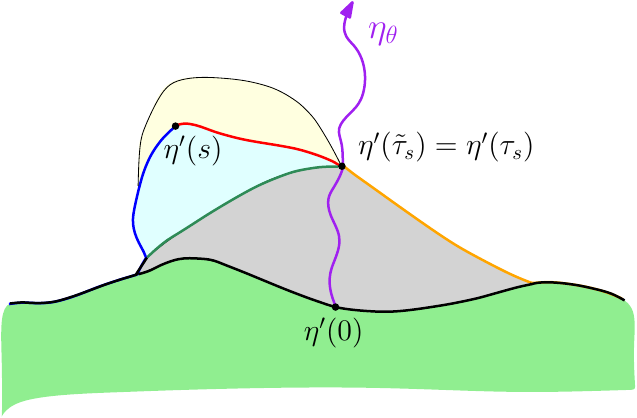}
	\caption{An illustration of Proposition \ref{prop:contour}. Fix $s>0$. The space-filling flow line $\eta'$ generated by the GFF $\hat{h}$, drawn on the top of an independent weight $4-\gamma^2$ quantum cone, first traces the green region during the time $(-\infty,0]$, then the gray region for time $t\in [0,\tilde{\tau}_s]$, and the blue region for time $t\in [\tilde{\tau}_s, s]$ and finally the yellow region for $t\in [s,\tau_s]$. Here, $\tilde{\tau}_s$ (resp.\ ${\tau}_s$) is the last time before (resp.\ first time after) $s$ when $\eta'$ crosses  $\eta_{\theta}$. By definition,  $L_s-L_{\tilde{\tau}_s}$ is the quantum length of the blue segment minus the quantum length of the green segment, while $R_s-R_{\tilde{\tau}_s}=X_s$ is the quantum length of the red segment. It has been shown in \cite[Lemma 3.6]{GHS16} that a.s.\ $\eta'(\tilde{\tau}_s) = \eta'({\tau}_s)$. }\label{fig-contour-decomposition}
\end{figure}

Based on \cite[Lemma 3.17]{GHS16}, it has been proved in \cite{borga2021skewperm} that the process $(X_s)_{s\ge 0}$ solves the following SDE
\begin{equation}\label{eqn:X-sde}
    dX_s = \mathds{1}_{X_s>0}dR_s - \mathds{1}_{X_s<0}dL_s + (2p_\gamma(\theta)-1)d\mathcal{L}_s; \ \ s\ge 0
\end{equation}
where $(\mathcal{L}_s)_{s\ge0}$ is the local time of $|X|$ at 0. Moreover, the solution to \eqref{eqn:X-sde} a.s.\ exists and is unique. Based on this, $X$ has the law of a skew Brownian motion with parameter $p$ as in \cite{Lej06}. Moreover, we have
\begin{proposition}[Proposition 3.3 of \cite{GHS16}]\label{prop:X-local-time}
In the setting of Proposition \ref{prop:contour} and \eqref{eqn:X-sde}, there exists a deterministic constant $c_\gamma(\theta)>0$, such that a.s.\ for every $t\ge 0$, $\frac{1}{2}\mathcal{L}_t$ is $c_\gamma(\theta)$ times the quantum length of $\eta_\theta\cap\eta'([0,t])$.
\end{proposition}

As explained after~\cite[Proposition 3.3]{GHS16}, the reason for the factor of $\frac{1}{2}$ is so that $(\frac{1}{2}\mathcal{L}_t)_{t\geq0}$ is the local time of $X$ at 0 when $\theta=0$, in which case $X$ is a Brownian motion.

The following monotonicity is clear from the fact that $p_\gamma(\theta)$ is the probability that $\eta'(1)$ lies to the left of $\eta_\theta$, which is essentially a consequence of the monotonicity of imaginary geometry flow lines in~\cite{MS16a}. 
\begin{lemma}\label{lem:dec}
    The constant $p_\gamma(\theta)$ is decreasing for $\theta\in(-\frac{\pi}{2},\frac{\pi}{2})$.
\end{lemma}
\begin{proof}
    {By~\cite[Theorem 1.5]{MS16a}, if $\frac{\pi}{2}>\theta_1>\theta_2>-\frac{\pi}{2}$, then $\eta_{\theta_1}$ is lying on the left of $\eta_{\theta_2}$.} Therefore    $p_\gamma(\theta)$ is decreasing for $\theta\in(-\frac{\pi}{2},\frac{\pi}{2})$.
\end{proof}

\subsection{The constant $p$ in terms of quantum area}\label{sec:p-quantum-area}

Recall the setting of Propositions \ref{prop:contour} and \ref{prop:X-local-time}. For $t>0$, let $\sigma_t$ be the first time when the quantum length of $\eta_\theta\cap\eta'([0,t])$ is $t$, and $A_t^\rmL$ (resp.\ $A_t^\rmR$) be the corresponding quantum area of the part of $\eta'([0,\sigma_t])$ lying on the left (resp.\ right) of $\eta_\theta$. By Proposition \ref{prop:X-local-time}, a.s.\ for all $t>0$ we have $\mathcal{L}_{\sigma_t} = 2c_\gamma(\theta)t$.

\begin{proposition}\label{prop:p-quantum-area}
Consider  the setting of Proposition \ref{prop:contour}. Let $\mathbf{a}$ be given in~\eqref{eq:mot-var-a} and let the constants $p_\gamma(\theta)$ and $c_\gamma(\theta)$ be as in Proposition~\ref{prop:contour} and Proposition~\ref{prop:X-local-time}. For any $\mu>0$, we have
\begin{equation}\label{eqn:p-quantum-area}
     \mathbb{E}[\exp(-\mu A_1^\rmL)] = \exp\big(-\bfa^{-1} c_\gamma(\theta)p_\gamma(\theta)\sqrt{2\mu}\big), \ \ \mathbb{E}[\exp(-\mu A_1^\rmR)] = \exp\big(-\bfa^{-1} c_\gamma(\theta)(1-p_\gamma(\theta))\sqrt{2\mu}\big).
\end{equation}
\end{proposition}

\begin{proof}
	
	Since the local time of $|X|$ at time $\sigma_t$ is $2c_\gamma(\theta)t$ (Proposition \ref{prop:X-local-time}), by \cite[Theorem 19.13]{Kal97}
	\begin{equation}\label{eq:Poissonian}
	\sigma_t = \int_0^{2c_\gamma(\theta)t}\int \ell(u)\xi(dr\,du)
	\end{equation}
	for a Poisson point process  $\xi$ on $\bbR_+\times \mathcal{E}_0$ where $\mathcal{E}_0$ is the set of excursions and $\ell(u)$ is the length of the excursion for $u\in \mathcal{E}_0$.   By Proposition \ref{prop:contour},   since $\eta'$ is parameterized by quantum area, 
    the jumps of $A_t^L$ (resp.\ $A_t^R$) 
    are precisely  {the lengths of the excursions $u$ which we keep in the Poissonian decomposition~\eqref{eq:Poissonian} with probability $p_\gamma(\theta)$ (resp. $1-p_\gamma(\theta)$) independently.} Therefore by the thinning of Poisson point processes, $A_t^\rmL$ and $A_t^\rmR$ are independent, with the same distributions as $\sigma_{p_\gamma(\theta)t}$ and $\sigma_{(1-p_\gamma(\theta))t}$ respectively. Now $\sigma_t$ is the first time when the local time of $|X|$ at 0 is $2c_\gamma(\theta)t$,  which has the same distribution as the first time when the local time of a standard Brownian motion $(B_t)_{t\geq0}$ at 0 is $\bfa^{-1} c_\gamma(\theta)t$.  Therefore by~\cite[Theorem 9.14]{LGbook}, $\sigma_t$ has the same distribution as $\inf\{s>0: B_s = \bfa^{-1} c_\gamma(\theta)t\}$, and further by scaling, if we set $T_1 = \inf\{s>0:B_s=1\}$, then 
    $$\sigma_t \overset{d}{=} \bfa^{-2} c_\gamma(\theta)^2t^2 T_1, \ \ \  A_1^\rmL \overset{d}{=} \bfa^{-2} c_\gamma(\theta)^2p_\gamma(\theta)^2 T_1, \ \ \  A_1^\rmR \overset{d}{=} \bfa^{-2} c_\gamma(\theta)^2(1-p_\gamma(\theta))^2T_1.$$  {The claim~\eqref{eqn:p-quantum-area} then follows} from the fact that for $\lambda>0$, $\bbE \exp(-\lambda T_1)=\exp(-(2\lambda)^{1/2})$.
\end{proof}

Next we link $A_1^\rmL$ and $A_1^\rmR$ to areas of quantum triangles.  See Figure~\ref{fig-weld-wedge}.   {Recall the measure $\QT(W^\rmL,2-W^\rmL,\gamma^2;\ell)$  from the disintegration of   $\QT(W^\rmL,2-W^\rmL,\gamma^2)$. Following~\cite[Proposition 2.24]{ASY22}, $|\QT(W^\rmL,2-W^\rmL,\gamma^2;\ell)|<\infty$, and we write $\QT(W^\rmL,2-W^\rmL,\gamma^2;1)^\#$ for $\QT(W^\rmL,2-W^\rmL,\gamma^2;1)/|\QT(W^\rmL,2-W^\rmL,\gamma^2;1)|$.}
\begin{proposition}\label{prop:p-qt}
 Let $W^\rmL = (1-\frac{\gamma^2}{4})(1-\frac{2\theta}{\pi})$ and $W^\rmR = (1-\frac{\gamma^2}{4})(1+\frac{2\theta}{\pi})$. Assume $W^\rmL \neq\frac{\gamma^2}{2}$.  Then the random variable $A_1^\rmL$   has the same law as the quantum area of a quantum triangle from $\QT(W^\rmL,2-W^\rmL,\gamma^2;1)^\#$  where the boundary arc that is not adjacent to the weight $\gamma^2$ vertex of the quantum triangle has quantum length 1. The analogous statement holds for $A_1^\rmR$ and $W^\rmR$ instead of $A_1^\rmL$ and $W^\rmL$,  {if we assume $(W^\rmL,W^\rmR)\notin \bbR\times \{\frac{\gamma^2}2\}$.}
\end{proposition}

\begin{figure}[ht]
	\centering
	\includegraphics[width=0.5\textwidth]{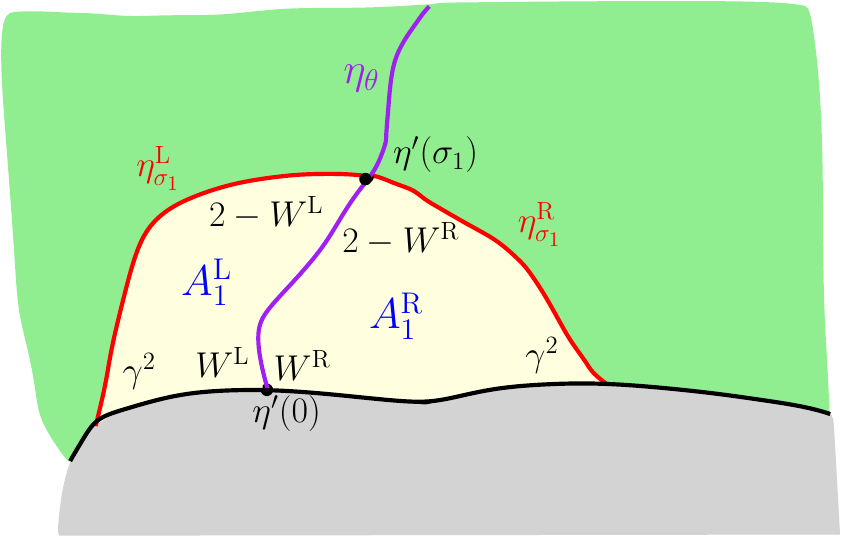}
	\caption{Illustration of Proposition \ref{prop:p-qt}. The segment of $\eta_\theta$ from $\eta(0)$ to $\eta(\sigma_1)$ has quantum length 1,  and $A_1^\rmL$ and $A_1^\rmR$ are the areas of the   quantum triangles with the weights shown.     }\label{fig-weld-wedge}
\end{figure}

\begin{proof}
    Let $\eta_{\sigma_1}^\rmL$ and   $\eta_{\sigma_1}^\rmR$ be the left and right boundaries of $\eta((-\infty,\sigma_1])$. Let  $\bbH^\rmL$ (resp.\ $\bbH^\rmR$) be the part of $\eta((0,\infty))$ lying on the left (resp.\ right) hand side of  $\eta_\theta$. Then  {following~\cite[Theorem 1.5]{DMS14}, $(\eta((0,\infty)),h,0,\infty)/{\sim_\gamma}$ is a weight $2-\frac{\gamma^2}{2}$ quantum wedge, and from the imaginary geometry flow line description in~\cite[Theorem 1.1]{MS16a},} $\eta_\theta$ is an $\SLE_\kappa(W^\rmL-2;W^\rmR-2)$ curve within $\eta((0,\infty))$. By  Theorem~\ref{thm:wedge-welding}, $(\bbH^\rmL,h,0,\infty)/{\sim_\gamma}$ and $(\bbH^\rmR,h,0,\infty)/{\sim_\gamma}$ are independent quantum wedges of weight $W^\rmL$ and $W^{\rmR}$. Furthermore, given $\partial\eta((-\infty,0))$ and $\eta_\theta$, by~\cite[Theorem 1.1 and Theorem 1.4]{MS16a}, $\eta_{\sigma_1}^\rmL$ is an $\SLE_\kappa(-W^\rmL,W^\rmL-2;W^\rmL-2)$ curve in $\bbH^\rmL$, and $\eta_{\sigma_1}^\rmR$ is an $\SLE_\kappa(W^\rmR-2;-W^\rmR,W^\rmR-2)$ curve in $\bbH^\rmR$. The claim then follows by disintegrating  over the boundary length of the arc between the third marked point and the finite endpoint of the quantum wedge on the left hand side of~\eqref{eq:thm-weld-0} in Theorem~\ref{thm:weld-0}.  
\end{proof}

 To remove the conditioning on the boundary quantum length in Proposition~\ref{prop:p-qt},  {i.e., the length 1 in $\QT(2-W,W,\gamma^2;1)^\#$,} we need the following scaling property of quantum triangles.
 \begin{lemma}\label{lm:qt-scaling}
     For $W_1,W_2,W_3,\ell>0$, the following procedures agree for all $\lambda>0$:
     \begin{itemize}
         \item Sample a quantum triangle from $\QT(W_1,W_2,W_3;\lambda\ell)$;
         \item Sample a quantum triangle from $\lambda^{\frac{\gamma^2+2-(W_1+W_2+W_3)}{\gamma^2}}\QT(W_1,W_2,W_3;\ell)$ and add $\frac{2}{\gamma}\log \lambda$ to its field.
     \end{itemize}
 \end{lemma}
 The proof, which is based on changing variables in the integration over the constant term $\mathbf{c}\sim e^{(\frac{\beta_1+\beta_2+\beta_3}{2}-Q)c}dc$, is identical to \cite[Lemma 2.18 and Lemma 2.23]{AHS20}. We omit the details.

 \begin{lemma}\label{lem:p-qt-A}
  In the setting of Proposition~\ref{prop:p-qt}, let $\beta^\rmL = \gamma+\frac{2-W^\rmL}{\gamma}$, $\wt{\beta^\rmL} = \gamma+\frac{W^\rmL}{\gamma}$, $\beta^\rmR = \gamma+\frac{2-W^\rmR}{\gamma}$ and $\wt{\beta^\rmR} = \gamma+\frac{W^\rmR}{\gamma}$.   If $W^\rmL\neq\frac{\gamma^2}{2}$, then
  \begin{equation}\label{eq:lem:p-qt-A}
       p_\gamma(\theta)c_\gamma(\theta) = \frac{\bfa}{\sqrt 2}\cdot\frac{|\QT(W^\rmL,2-W^\rmL,\gamma^2;1)|}{\QT(W^\rmL,2-W^\rmL,\gamma^2)[e^{-A}]}  = \frac{2}{\gamma\sqrt{\sin(\frac{\pi\gamma^2}{4})}}\cdot\frac{\bar{H}_{(0,1,0)}^{(\beta^\rmL,\wt{\beta^\rmL},\frac{2}{\gamma})}}{{H}_{(0,0,0)}^{(\beta^\rmL,\wt{\beta^\rmL},\frac{2}{\gamma})}}. 
  \end{equation}
  where $A$ is the area of the quantum triangle,  $\bar{H}$ is defined in~\eqref{eq-bar-H} and $H$ is defined in Definition~\ref{def-H} and Proposition~\ref{prop-H-extend}.
  Likewise,  {if $W^\rmR\neq\frac{\gamma^2}{2}$, then}
  \begin{equation}\label{eq:lem:p-qt-B}
     \ (1-p_\gamma(\theta))c_\gamma(\theta) = \frac{\bfa}{\sqrt 2}\cdot\frac{|\QT(W^\rmR,2-W^\rmR,\gamma^2;1)|}{\QT(W^\rmR,2-W^\rmR,\gamma^2)[e^{-A}]}   = \frac{2}{\gamma\sqrt{\sin(\frac{\pi\gamma^2}{4})}}\cdot\frac{\bar{H}_{(0,1,0)}^{(\beta^\rmR,\wt{\beta^\rmR},\frac{2}{\gamma})}}{{H}_{(0,0,0)}^{(\beta^\rmR,\wt{\beta^\rmR},\frac{2}{\gamma})}}.
  \end{equation}
 \end{lemma}

 \begin{proof}
 For a sample from $\QT(W^\rmL,2-W^\rmL,\gamma^2)$, let $A$ be its quantum area. Then by Lemma \ref{lm:qt-scaling}, we have $\QT(W^\rmL,2-W^\rmL,\gamma^2;\ell)^\#[e^{-\mu A}] = \QT(W^\rmL,2-W^\rmL,\gamma^2;1)^\#[e^{-\mu\ell^2 A}] $. Furthermore, by Proposition~\ref{prop:QT-bdry-length}, $\big|\QT(W^\rmL,2-W^\rmL,\gamma^2;\ell)\big|=\big|\QT(W^\rmL,2-W^\rmL,\gamma^2;1)\big|$. Therefore, using Proposition~\ref{prop:p-qt}, \begin{equation}\label{eqn:p-lcft-a}
     \begin{split}
         \QT(W^\rmL,2-W^\rmL,\gamma^2)&[e^{-  A}] = \int_0^\infty \big|\QT(W^\rmL,2-W^\rmL,\gamma^2;\ell)\big|\cdot\QT(W^\rmL,2-W^\rmL,\gamma^2;\ell)^\#[e^{- A}]d\ell\\
         &= \int_0^\infty \big|\QT(W^\rmL,2-W^\rmL,\gamma^2;1)\big|\cdot\QT(W^\rmL,2-W^\rmL,\gamma^2;1)^\#[e^{-\ell^2 A}]d\ell\\
		&= \int_0^\infty \big|\QT(W^\rmL,2-W^\rmL,\gamma^2;1)\big|\cdot \bbE[e^{- \ell^2 A_1^\rmL}] d\ell\\
		&= \int_0^\infty \big|\QT(W^\rmL,2-W^\rmL,\gamma^2;1)\big|\cdot\exp\big(-\bfa^{-1} c_\gamma(\theta)p_\gamma(\theta)\sqrt{2}\ell\big)d\ell \\
		&= \big|\QT(W^\rmL,2-W^\rmL,\gamma^2;1)\big|\cdot\frac{\bfa}{\sqrt{2}c_\gamma(\theta)p_\gamma(\theta)}.
     \end{split}
 \end{equation}
This verifies the first half of~\eqref{eq:lem:p-qt-A}.  Using the value of $\mathbf{a}$ from~\eqref{eq:mot-var-a},  the second half of equation~\eqref{eq:lem:p-qt-B} is a direct consequence of~\eqref{eq:lem:p-qt-A} along with Proposition~\ref{prop:QT-bdry-length}, Lemma~\ref{lem-H-thick} and Lemma~\ref{lem-H-thin}. The equation~\eqref{eq:lem:p-qt-B} follows similarly.
\end{proof}

\subsection{Computation of quantum area from Liouville CFT}\label{subsec:area-lcft}

Recall the functions $H^{(\beta_1, \beta_2, \beta_3)}_{(\mu_1, \mu_2, \mu_3)}$ and $H\begin{pmatrix} \beta_1, \beta_2, \beta_3 \\ \sigma_1,  \sigma_2,   \sigma_3 \end{pmatrix}$ introduced in Section~\ref{sec:lcft-qt}. The goal of this section is to prove the following identity. 

\begin{proposition}\label{prop-H-special}
  {As meromorphic functions in $\beta$,} we have
    \[H^{(Q + \frac{3\gamma}2 - \beta, \frac2\gamma, \beta)}_{(0,0,0)} = \frac{2}{\gamma \pi} \Gamma(1-\frac{\gamma^2}4)^2 \sqrt{\sin (\frac{\pi\gamma^2}4)}
			\frac
			{\Gamma(\frac\gamma2(Q+\frac\gamma2 - \beta)) \Gamma(\frac\gamma2(\beta-\gamma))}
			{ \Gamma(\frac\gamma2(Q - \beta)) 
			\Gamma(\frac\gamma2(\beta - \frac{3\gamma}2))
			}. \]
\end{proposition}

The key input is the following shift equation for $H$. 
	
	\begin{lemma}[{\cite[Theorem 3.1]{arsz-structure-constants}}]\label{lem-full_shift_H} Let $\chi = \frac{\gamma}{2}$ or $\frac{2}{\gamma}$ and fix $\sigma_3 \in [-\frac{1}{2 \gamma} + \frac{Q}{2}, \frac{1}{2 \gamma} + \frac{Q}{2} - \frac{\gamma}{4}] \times \mathbb{R} $. Writing $q = \frac1\gamma(2Q - \beta_1 - \beta_2 - \beta_3 + \chi)$ and $g_\chi(\sigma) = \left( \sin(\frac{\pi\gamma^2}4)\right)^{-\chi/\gamma} \cos\left( 2 \pi \chi (\sigma - \frac Q2)\right)$, we have for all $\beta_1, \beta_2, \beta_3, \sigma_1, \sigma_2$ that
		\begin{align*}
			&  \frac{\chi^2 \pi^{\frac{2\chi}{\gamma} -1 }  }{\Gamma(1-\frac{\gamma^2}{4})^{\frac{2\chi}{\gamma}}}  \Gamma(1-\chi\beta_2)\left( g_{\chi}(\sigma_3) - g_{\chi}(\sigma_2+\frac{\beta_2}{2}) \right) 
			H
			\begin{pmatrix}
				\beta_1 , \beta_2 + \chi, \beta_3 \\
				\sigma_1,  \sigma_2 + \frac{\chi}{2},   \sigma_3 
			\end{pmatrix}\\
     &=\frac{\Gamma(\chi(\beta_1-\chi))}{\Gamma( - q \frac{\gamma\chi}{2} ) \Gamma(-1 + \chi(\beta_1 + \beta_2-2\chi+q\frac{\gamma}{2}) ) } H
			\begin{pmatrix}
				\beta_1 - \chi, \beta_2, \beta_3 \\
				\sigma_1,  \sigma_2,   \sigma_3 
			\end{pmatrix}\\
			&+   \frac{\chi^2 \pi^{\frac{2\chi}{\gamma}} }{\Gamma(1-\frac{\gamma^2}{4})^{\frac{2\chi}{\gamma}}}  \frac{   \left(  g_{\chi}(\sigma_1) - g_{\chi}(\sigma_2-\frac{\beta_1}{2}+\frac{\chi}{2} ) \right) \Gamma(1-\chi\beta_1) }{\sin(\pi\chi(\chi-\beta_1))\Gamma(1- \chi(\beta_1-\chi+q\frac{\gamma}{2} )) \Gamma( \chi\beta_2 - \chi^2 +q\frac{\gamma\chi}{2}) }  H
			\begin{pmatrix}
				\beta_1 + \chi, \beta_2, \beta_3 \\
				\sigma_1,  \sigma_2,   \sigma_3 
			\end{pmatrix}.
		\end{align*}
	\end{lemma}	
    The statement of \cite[Theorem 3.1]{arsz-structure-constants} differs from 
    Lemma~\ref{lem-full_shift_H} in that it has the additional assumption that $\gamma \neq \sqrt2$; as explained at the start of Section 3 of \cite{arsz-structure-constants}, this assumption can be removed given \cite[Theorem 1.1]{arsz-structure-constants}.
    
	We fix $\chi = \frac\gamma2$ so $g_\chi(\sigma) = \left( \sin(\frac{\pi\gamma^2}4)\right)^{-1/2} \cos\left( 2 \pi \chi (\sigma - \frac Q2)\right)$, and set $\sigma_1 = \sigma_2 = \sigma_3 = -\frac1{2\gamma} + \frac Q2$ so $g_\chi(\sigma_i) = 0$ for $i = 1,2,3$. Then, Lemma~\ref{lem-full_shift_H} gives
    \begin{align}\label{eq-shift-2}
		&-\frac{(\frac\gamma2)^2}{\Gamma(1-\frac{\gamma^2}{4})}  \Gamma(1-\frac\gamma2\beta_2)  g_{\frac\gamma2}(-\frac1{2\gamma} + \frac Q2+\frac{\beta_2}{2}) 
		H^{(\beta_1, \beta_2+ \frac\gamma2, \beta_3)}_{(0,\mu_2,0)}\\
        &=\frac{\Gamma(\frac\gamma2(\beta_1-\frac\gamma2))}{\Gamma( - q \frac{\gamma^2}{4} ) \Gamma(-1 + \frac\gamma2(\beta_1 + \beta_2-\gamma+q\frac{\gamma}{2}) ) } H^{(\beta_1 - \frac\gamma2, \beta_2, \beta_3)}_{(0,0,0)} \nonumber\\
	&- \frac{(\frac\gamma2)^2 \pi }{\Gamma(1-\frac{\gamma^2}{4})}  \frac{  g_{\frac\gamma2}(-\frac1{2\gamma} + \frac Q2-\frac{\beta_1}{2}+\frac{\gamma}{4} ) \Gamma(1-\frac\gamma2\beta_1) }{\sin(\pi\frac\gamma2(\frac\gamma2-\beta_1))\Gamma(1- \frac\gamma2(\beta_1-\frac\gamma2+q\frac{\gamma}{2} )) \Gamma( \frac\gamma2\beta_2 - (\frac\gamma2)^2 +q\frac{\gamma^2}{4}) }  H^{(\beta_1+ \frac\gamma2, \beta_2, \beta_3)}_{(0,0,0)}. \nonumber
    \end{align}
	where $\mu_2 = g_{\frac\gamma2}(\sigma_2 + \frac\gamma4) = (\sin(\frac{\pi\gamma^2}4))^{-1/2} \cos(\pi (\frac{\gamma^2}4 - \frac12))> 0$. 

    To prove Proposition~\ref{prop-H-special}, we will start with~\eqref{eq-shift-2} and apply the following two lemmas which give certain limiting values of $H$ via its probabilistic definition (Definition~\ref{def-H}).
	
	\begin{lemma}\label{lem-limit-beta_2}
		Suppose $\beta_1, \beta_3 < Q$  satisfy $\beta_1+\beta_3 > Q$ and $|Q - \beta_1|< \beta_3$. Then
		\[\lim_{\beta_2 \uparrow \frac2\gamma}  \Gamma(1-\frac\gamma2\beta_2)  g_{\frac\gamma2}(-\frac1{2\gamma} + \frac Q2+\frac{\beta_2}{2}) 
		H^{(\beta_1, \beta_2+ \frac\gamma2, \beta_3)}_{(0,\mu_2,0)}  = 0.\]
	\end{lemma}
	\begin{proof}
		At the point $\beta_2 = \frac2\gamma$, the meromorphic function $\Gamma(1 - \frac\gamma2 \beta_2)$ has a simple pole whereas $g_{\frac\gamma2}(-\frac1{2\gamma} + \frac Q2+\frac{\beta_2}{2}) = (\sin(\frac{\pi\gamma^2}4))^{-1/2}   {\cos(-\frac\pi2 + \frac{\pi\gamma \beta_2}{2})}$ has a simple root. Thus it suffices to show $\lim_{\beta_2 \uparrow \frac2\gamma} 	H^{(\beta_1, \beta_2+ \frac\gamma2, \beta_3)}_{(0,\mu_2,0)} = 0$. By Definition~\ref{def-H} and Lemma~\ref{lm:bdry-length-law}, 
        \begin{equation*}
            \begin{gathered}
                H^{(\beta_1, \beta_2+ \frac\gamma2, \beta_3)}_{(0,\mu_2,0)} <  \LF^{(\beta_1, +\infty), (\beta_2 + \frac\gamma2, -\infty), (\beta_3, 0)}[e^{-\mu_2 \nu_\phi(\partial_2\mathcal S)}] =  \int_0^\infty \frac2\gamma \bar{H}^{(\beta_1, \beta_2+ \frac\gamma2, \beta_3)}_{(0,1,0)} \ell^{(\sum \beta_i +\frac\gamma2 - 2Q)/\gamma - 1} e^{-\mu_2 \ell}\, d\ell \\
           \qquad \qquad \qquad \qquad \qquad \qquad \qquad \qquad \qquad \qquad \qquad =  \frac2\gamma \bar{H}^{(\beta_1, \beta_2+ \frac\gamma2, \beta_3)}_{(0,1,0)} \Gamma(s)\mu_2^{-s}, \quad s = \frac1\gamma(\sum \beta_i + \frac\gamma2 - 2Q).
            \end{gathered}
        \end{equation*}
        The last equality follows from the integral definition of the Gamma function. We have $\lim_{\beta_2 \uparrow \frac2\gamma}\Gamma(s)\mu_2^{-s} = \Gamma(s_0) \mu_2^{s_0}$ where   {$s_0 = \frac{\beta_1 + \beta_3 - Q}{\gamma}$}, but $\lim_{\beta_2 \uparrow \frac2\gamma} \bar{H}^{(\beta_1, \beta_2+ \frac\gamma2, \beta_3)}_{(0,1,0)} = 0$ due to the term $\Gamma_{\frac\gamma2} (Q-(\beta_2+\frac\gamma2))^{-1}$ in the formula for $\bar H$~\eqref{eq-bar-H} (see the text below~\eqref{eqn-double-gamma} for the poles of $\Gamma_{\frac\gamma2}$). We conclude that $\limsup_{\beta_2\uparrow \frac2\gamma} H^{(\beta_1, \beta_2+ \frac\gamma2, \beta_3)}_{(0,\mu_2,0)} \leq 0$. This matches the trivial lower bound $H^{(\beta_1, \beta_2+ \frac\gamma2, \beta_3)}_{(0,\mu_2,0)} \geq 0$, so $\lim_{\beta_2 \uparrow \frac2\gamma} 	H^{(\beta_1, \beta_2+ \frac\gamma2, \beta_3)}_{(0,\mu_2,0)} = 0$. 
	\end{proof}
	\begin{lemma}\label{lem-limit-s}
		Suppose $\beta_2, \beta_3 < Q$ and $2Q - \beta_2 - \beta_3 < Q$. Then
		\[ \lim_{\beta_1 \downarrow 2Q - \beta_2 - \beta_3} (\frac12 \sum_i \beta_i  - Q)H^{(\beta_1,\beta_2,\beta_3)}_{(0,0,0)} = 1.\]
    \end{lemma}
    \begin{proof}
        By Definitions~\ref{def-lf-strip} and~\ref{def-H} we have 
        \[H^{(\beta_1,\beta_2,\beta_3)}_{(0,0,0)} = \bbE\left[\int_\bbR \exp(-\mu_{\hat h + c}(\mathcal S))  e^{(\frac12 \sum \beta_1 - Q)c}\, dc\right] = \bbE\left[\int_\bbR \exp(-e^{\gamma c}\mu_{\hat h}(\mathcal S))  e^{(\frac12 \sum \beta_1 - Q)c}\, dc\right]\]
        where $\bbE$ represents the expectation taken over the free field $h$ on the strip $\mathcal S$, and $\hat h(z) = h(z) +\frac{\beta_1+\beta_2-2Q}{2}|\textup{Re}z|+\frac{\beta_1-\beta_2}{2}\textup{Re}z+\frac{\beta_3}{2}G_{\mathcal{S}}(z, s_3)$. Using the change of variables $y = e^{\gamma c} \mu_{\hat h}(\mathcal S)$, we obtain
        \[H^{(\beta_1,\beta_2,\beta_3)}_{(0,0,0)} = \bbE \left[ \int_0^\infty e^{-y} \left( \frac y {\mu_{\hat h}(\mathcal S)}\right)^{\frac1\gamma(\frac12\sum \beta_i - Q)} \frac1{\gamma y} \, dy \right] = \bbE[\mu_{\hat h}(\mathcal S)^{\frac1\gamma(Q - \frac12\sum\beta_i)}] \cdot \frac1\gamma \Gamma\left(\frac1\gamma(\frac12 \sum \beta_i - Q)\right). \]
        We have $\lim_{s \downarrow 0} s \Gamma(s) = \lim_{s \downarrow 0} \Gamma(s+1) = 1$, so it suffices to show $\bbE[\mu_{\hat h}(\mathcal S)^{\frac1\gamma(Q - \frac12\sum\beta_i)}]\to 1$. 

        Let $\beta_* = 2Q - \beta_2 - \beta_3$ and fix $\beta^* \in (Q - \beta_2 - \beta_3, Q)$, and consider $\beta_1 \in (\beta_*, \beta^*)$.
        Let $h_*$ be defined in the same way as $\hat h$ but with $\beta_1$ replaced by $\beta_* = 2Q - \beta_2 - \beta_3$, and similarly define $h^*$. We have $\mu_{h_*}(\mathcal S) < \mu_{\hat h}(\mathcal S) < \mu_{h^*}(\mathcal S)$ and $\frac1\gamma(Q - \frac12\sum_i \beta_i) < 0$, so 
        $ \bbE[\mu_{h^*}(\mathcal S)^{\frac1\gamma(Q - \frac12\sum\beta_i)}] \leq  \bbE[\mu_{\hat h}(\mathcal S)^{\frac1\gamma(Q - \frac12\sum\beta_i)}] \leq \bbE[\mu_{h_*}(\mathcal S)^{\frac1\gamma(Q - \frac12\sum\beta_i)}]$. Since $\bbE[\mu_{h_*}(\mathcal S)^{\frac1\gamma(Q - \frac12\sum\beta_i)}+1] < \infty$ \cite[Corollary 6.11]{hrv-disk} and $\lim_{\beta_1 \downarrow \beta_*}\frac1\gamma(Q - \frac12\sum\beta_i) = 0$,  the dominated convergence theorem implies the upper and lower bounds both converge to 1, hence $\lim_{ \beta_1 \downarrow \beta_*} \bbE[\mu_{\hat h}(\mathcal S)^{\frac1\gamma(Q - \frac12\sum\beta_i)}]  = 1$ as desired. 
\end{proof}

    We are now ready to prove Proposition~\ref{prop-H-special}. 

    \begin{proof}[Proof of Proposition~\ref{prop-H-special}]
     {We first prove the identity for $\beta \in (\gamma, Q)$.}
    Consider $\beta_1, \beta_3 < Q$ such that $\beta_1 + \beta_3 > Q$ and $|Q - \beta_1| < \beta_3$. Taking the $\beta_2 \uparrow \frac2\gamma$ limit gives zero in the left hand side of~\eqref{eq-shift-2}, hence
	
			\begin{align}\label{eq-shift-3}
		&\frac{\Gamma(\frac\gamma2(\beta_1-\frac\gamma2))}{\Gamma( - q \frac{\gamma^2}{4} ) \Gamma(-1 + \frac\gamma2(\beta_1 + \frac2\gamma-\gamma+q\frac{\gamma}{2}) ) } H^{(\beta_1 - \frac\gamma2, \frac2\gamma, \beta_3)}_{(0,0,0)}\\
		&=  \frac{(\frac\gamma2)^2 \pi }{\Gamma(1-\frac{\gamma^2}{4})}  \frac{  g_{\frac\gamma2}(-\frac1{2\gamma} + \frac Q2-\frac{\beta_1}{2}+\frac{\gamma}{4} ) \Gamma(1-\frac\gamma2\beta_1) }{\sin(\pi\frac\gamma2(\frac\gamma2-\beta_1))\Gamma(1- \frac\gamma2(\beta_1-\frac\gamma2+q\frac{\gamma}{2} )) \Gamma( 1 - (\frac\gamma2)^2 +q\frac{\gamma^2}{4}) }  H^{(\beta_1+ \frac\gamma2, \frac2\gamma, \beta_3)}_{(0,0,0)}. \nonumber
	\end{align}
	Set $\beta_3 = \beta \in (\gamma, Q)$, so~\eqref{eq-shift-3} holds when $\beta_1 \in (Q + \gamma - \beta, Q)$. By Lemma~\ref{lem-limit-s}, with $q = \frac1\gamma(Q + \gamma - \beta - \beta_1)$,
	\[\lim_{\beta_1 \downarrow Q + \gamma - \beta}\frac1{\Gamma(-q \frac{\gamma^2}4)} H^{(\beta_1 - \frac\gamma2, \frac2\gamma, \beta)}_{(0,0,0)} = \lim_{\beta_1 \downarrow Q + \gamma - \beta} \frac{-\gamma^2}{4 \Gamma(1-q \frac{\gamma^2}4)} q H^{(\beta_1 - \frac\gamma2, \frac2\gamma, \beta)}_{(0,0,0)} =  \frac{\gamma}2, \]
	so taking the limit as $\beta_1 \downarrow Q + \gamma -\beta$ (and so $q \to 0$)  in~\eqref{eq-shift-3}  gives
		\[
		\frac\gamma2\frac{\Gamma(\frac\gamma2(Q+\frac\gamma2 - \beta))}{\Gamma(\frac\gamma2(Q - \beta)) } =  \frac{(\frac\gamma2)^2 \pi }{\Gamma(1-\frac{\gamma^2}{4})}  \frac{  g_{\frac\gamma2}(-\frac1{2\gamma} + \frac Q2 + \frac12(\beta - Q - \frac\gamma2) ) \Gamma(\frac\gamma2(\beta - \frac{3\gamma}2)) }{\sin(\pi\frac\gamma2(\beta - Q - \frac\gamma2))\Gamma(\frac\gamma2(\beta-\gamma)) \Gamma( 1 - \frac{\gamma^2}4) }  H^{(Q + \frac{3\gamma}2 - \beta, \frac2\gamma, \beta)}_{(0,0,0)}.\]
	Rearranging and using 
	$g_{\frac\gamma2}(-\frac1{2\gamma} + \frac Q2 + \frac12(\beta - Q - \frac\gamma2)) = (\sin \frac{\pi\gamma^2}4)^{-1/2} \sin(\pi\frac\gamma2(\beta-Q-\frac\gamma2))$ gives the claim   {for $\beta \in (\gamma, Q)$.}

     {Finally, by Proposition~\ref{prop-H-extend} $H^{(Q + \frac{3\gamma}2 - \beta, \frac2\gamma, \beta)}_{(0,0,0)}$ is a meromorphic function in $\beta$, and since the gamma function is meromorphic, so is the right hand side of the desired identity. These two meromorphic functions agree on $(\gamma, Q)$, hence they agree on $\bbC$ as needed. 
    }
    \end{proof}

\subsection{The $p_\gamma(\theta)$ and $c_\gamma(\theta)$ formula}\label{subsec:p-theta-result}

We are now ready to compute the constant $p_\gamma(\theta)$  using ~\eqref{eq:lem:p-qt-B} and Proposition~\ref{prop-H-special}.

\begin{proposition}\label{prop:p-c-theta}
    In the setting of Proposition~\ref{prop:p-qt},  if $W^\rmL\neq \frac{\gamma^2}{2}$, then
    \begin{equation}\label{eq:p-c-theta-a}
        p_\gamma(\theta)c_\gamma(\theta) = \frac{\sin((\frac{1}{2}-\frac{\gamma^2}{8})(\pi-2\theta))}{\sin(\frac{\pi\gamma^2}{4})}. 
    \end{equation}
    Similarly, if $W^\rmR\neq \frac{\gamma^2}{2}$, then
    \begin{equation}\label{eq:p-c-theta-b}
        (1-p_\gamma(\theta))c_\gamma(\theta) = \frac{\sin((\frac{1}{2}-\frac{\gamma^2}{8})(\pi+2\theta))}{\sin(\frac{\pi\gamma^2}{4})}.
    \end{equation}
\end{proposition}
\begin{proof}
 {Equation~\eqref{eq-bar-H} expresses $\bar H$ in terms of ratios of double-gamma functions.} Using the shift equations in~\eqref{eqn-gamma-shift}  {to simplify these ratios, we obtain that} 
    for $\beta\in(\frac{3\gamma}{2},Q+\frac{\gamma}{2})\backslash\{Q\}$ and $\wt{\beta} = Q+\frac{3\gamma}{2}-\beta$, \begin{equation}\label{eq:bar-H-beta}
        \bar{H}_{(0,1,0)}^{(\beta,\wt{\beta},\frac{2}{\gamma})} = \frac{\Gamma(1-\frac{\gamma^2}{4})^2}{\Gamma(\frac{\gamma}{2}(Q-\beta))\Gamma(\frac{\gamma}{2}(Q-\wt{\beta}))}.
    \end{equation}
    Also observe that by Lemma~\ref{lem-H-thick} and Lemma~\ref{lem-H-thin}, $H_{(0,0,0)}^{(\beta_1,\beta_2,\beta_3)}$ is symmetric over $(\beta_1,\beta_2,\beta_3)$. 
     {Therefore, we obtain formulae for $p_\gamma(\theta)c_\gamma(\theta)$ and $(1-p_\gamma(\theta))c_\gamma(\theta)$} 
    by setting $\beta = \beta^\rmL$ and $\beta=\beta^\rmR$ in~\eqref{eq:lem:p-qt-A},~\eqref{eq:lem:p-qt-B},~\eqref{eq:bar-H-beta} and Proposition~\ref{prop-H-special}, together with the formula
    $$\Gamma(z)\Gamma(1-z)=\frac{\pi}{\sin(\pi z)}, \ \ z\notin\mathbb{Z}.$$
     {To obtain~\eqref{eq:p-c-theta-a} and~\eqref{eq:p-c-theta-b}, we rewrite these formulae by expressing $\beta^\textup{L}$ and $\beta^\textup{R}$ in terms of $W^\textup{L}$ and $W^\textup{R}$ as in Lemma~\ref{lem:p-qt-A}, which are in turn expressed in terms of $\theta$ in Lemma~\ref{prop:p-qt}.}
\end{proof}


\begin{proof}[Proof of Theorem~\ref{thm:main} and Theorem~\ref{thm:main-1}]
    For $\theta\in(-\frac{\pi}{2},\frac{\pi}{2})$ satisfying the constraints in Proposition~\ref{prop:p-qt},  ~\eqref{eq:p-theta} and~\eqref{eq:c-theta} hold by Proposition~\ref{prop:p-c-theta}. For $\theta$ such that $W^\rmL=\frac{\gamma^2}{2}$ or $W^\rmR=\frac{\gamma^2}{2}$, the formula of $p_\gamma(\theta)$ in ~\eqref{eq:p-theta} follows from the monotonicity of $p_\gamma(\theta)$ in Lemma~\ref{lem:dec}. If $\gamma^2\neq\frac{4}{3}$, then it cannot hold that $W^\rmL=W^\rmR=\frac{\gamma^2}{2}$. Assume $W^\rmL\neq\frac{\gamma^2}{2}$. Then the formula of $c_\gamma(\theta)$ in~\eqref{eq:c-theta} follows from the expression of $p_\gamma(\theta)$ as well as~\eqref{eq:p-c-theta-a}. If $\gamma^2=\frac{4}{3}$ and $W^\rmL=W^\rmR=\frac{\gamma^2}{2}$, then~\cite[Corollary 5.6]{GHS16} states that $c_\gamma(\theta)=2$, matching the expression in~\eqref{eq:c-theta}.  {Finally, the case $\theta=\frac{\pi}{2}$ (resp.\ $\theta=-\frac{\pi}{2}$) is trivial since $\eta_\theta$ will be the left (resp.\ right) boundary of $\eta'([0,\infty))$ and the desired probability will be 0 (resp.\ 1).}
\end{proof}



\textbf{Acknowledgments.} 
We thank Scott Sheffield for raising this question in 2015. We thank Ewain Gwynne and Jacopo Borga for helpful communication on their related ongoing project. The second named author would like to  thank  Ewain Gwynne, Nina Holden, Yiting Li, and Sam Watson for collaborations on earlier related projects. 
M.A.\ was partially supported by NSF grant DMS-2348201. X.S.\ was partially supported by National Key R\&D Program of China (No.\ 2023YFA1010700). P.Y. was partially supported by NSF grant DMS-1712862.

\bibliographystyle{alpha}
\bibliography{theta}

\end{document}